\documentclass[aos]{imsart}
\RequirePackage{amsthm,amsmath,amsfonts,amssymb}
\RequirePackage[authoryear]{natbib} 
\RequirePackage[colorlinks,citecolor=blue,urlcolor=blue]{hyperref}
\RequirePackage{graphicx}
\usepackage{amsmath,amssymb,comment,enumerate}
\usepackage{amsfonts}
\usepackage{multirow}
\usepackage{mathrsfs}
\usepackage{url}
\usepackage{booktabs}
\usepackage{algorithm, algorithmic}
\usepackage{xr}
\usepackage{caption}
\usepackage{tikz}
\usepackage{subcaption}

\startlocaldefs
\theoremstyle{plain}
\newtheorem{theorem}{Theorem}[section]
\newtheorem{corollary}{Corollary}[section]
\newtheorem{lemma}[theorem]{Lemma}

\theoremstyle{remark}
\newtheorem{definition}{Definition}
\newtheorem{assumption}{Assumption}

\def\cN{\mathcal{N}}
\def\cS{\mathcal{S}}

\endlocaldefs
\begin{document}
	
	\begin{frontmatter}
		\title{Deep Nonparametric Regression on Approximate 
Manifolds: Non-Asymptotic Error Bounds with Polynomial Prefactors}
		\runtitle{Deep Nonparametric Regression}
		
		\begin{aug}
\author[A]{\fnms{Yuling} \snm{Jiao}\thanks{Yuling Jiao and Guohao Shen contributed equally to this work.}\ead[label=e1,mark]{yulingjiaomath@whu.edu.cn}},
\author[B]{\fnms{Guohao} \snm{Shen}$^*$\ead[label=e2,mark]{guohao.shen@polyu.edu.hk}}
\author[C]{\fnms{Yuanyuan} \snm{Lin}\ead[label=e3,mark]{ylin@sta.cuhk.edu.hk}}
			\and
\author[D]{\fnms{Jian} \snm{Huang}\ead[label=e4,mark]{j.huang@polyu.edu.hk}}
			\address[A]{School of Mathematics and Statistics,
				Wuhan University, Wuhan, Hubei, China\\
\printead{e1}}
			
			\address[B]{Department of Applied Mathematics, The Hong Kong Polytechnic University, Hong Kong SAR, China\\  \printead{e2}}	
				
				\address[C]{Department of Statistics, The Chinese University of Hong Kong,  Hong Kong SAR, China\\ \printead{e3}}	
			
\address[D]{Department of Applied Mathematics, The Hong Kong Polytechnic University, Hong Kong SAR, China\\  \printead{e4}}
		\end{aug}

\begin{abstract}
We study the properties of nonparametric least squares regression using deep neural networks. We derive non-asymptotic upper bounds for the prediction error of the empirical risk minimizer of feedforward deep neural regression. Our error bounds achieve minimax optimal rate and
improve over the existing ones in the sense that they depend polynomially on the dimension of the predictor, instead of exponentially on dimension. We show that the neural regression estimator can circumvent the curse of dimensionality under the assumption that the predictor is supported on an approximate low-dimensional manifold or a set with low Minkowski dimension.  We also establish the optimal convergence rate under the exact manifold support assumption. We investigate how the prediction error of the neural regression estimator depends on the structure of neural networks and propose a notion of network relative efficiency between two types of neural networks, which provides a quantitative measure for evaluating the relative merits of different network structures. To establish these results, we derive a novel approximation error bound for the H\"older smooth functions 
 using ReLU activated neural networks, which may be of independent interest. Our results are derived under weaker assumptions on the data distribution and the neural network structure than those in the existing literature.
\end{abstract}



\end{frontmatter}

\section{Introduction}
Consider a nonparametric regression model
\begin{equation}\label{model}
		Y=f_0(X)+\eta,
\end{equation}
where $Y \in \mathbb{R}$ is a response, $X \in \mathbb{R}^d$ is a $d$-dimensional vector of predictors, $f_0:[0,1]^d\to \mathbb{R}$ is an unknown regression function, $\eta$ is an error with mean 0 and finite variance $\sigma^2$, independent of $X$.
A basic problem in statistics and machine learning is to estimate the unknown target regression function $f_0$ based on a random sample,  $(X_i, Y_i), i=1, \ldots, n,$ where $n$ is the sample size,
that are independent and identically distributed  (i.i.d.) as
$(X, Y).$

There is a vast literature on nonparametric regression based on
minimizing the empirical least squares loss function,
see, for example,   \citet{nemirovskij1985rate}, \cite{van1990estimating},  \citet{birge1993rates} and the references therein. The consistency of the nonparametric least squares estimators under general conditions was  studied by \citet{geman1982nonparametric}, \citet{nemirovski1983estimators}, \citet{nemirovski1984signal},  \citet{van1987new} and \citet{van1996consistency}, among others.
In the context of pattern recognition, comprehensive results concerning empirical risk minimization can be found in  \cite{devroye2013probabilistic} and  \cite{gyorfi2006distribution}. 
In addition to the consistency, the convergence rate of the empirical risk minimizers
was analyzed in many important works. Examples include \citet{stone1982optimal}, \citet{pollard2012convergence}, \citet{rafaj1987nonparametric}, \citet{cox1988approximation}, \citet{shen1994convergence}, \citet{lee1996efficient}, \citet{birge1998minimum} and  \citet{geer2000empirical}.
These results were generally established under certain smoothness assumption on the unknown target function $f_0$. Typically, it is assumed that $f_0$ is in a H\"older class with a smoothness index $\beta>0$ ($\beta$-H\"older smooth), i.e., all the partial derivatives up to order $\lfloor\beta\rfloor$ exist and the partial derivatives of order $\lfloor\beta\rfloor$ are $\beta-\lfloor\beta\rfloor$ H\"older continuous, where $\lfloor\beta\rfloor$ denotes the largest integer strictly smaller than $\beta$.
For such an $f_0$, the optimal convergence rate of the prediction error is $C_d n^{-2\beta/(2\beta+d)}$
under mild conditions \citep{stone1982optimal}, where $C_d$ is a prefactor independent of $n$ but depending on $d$ and other model parameters. In low-dimensional models with a small $d$, the impact of $C_d$ on the convergence rate is not significant, however, in high-dimensional models with a large $d$, the impact of $C_d$ can be substantial, see, for example, \citet{ghorbani2020discussion}. Therefore, it is crucial to elucidate
how this prefactor depends on the dimensionality so that the error bounds are meaningful
in the high-dimensional settings.
	
Recently, several elegant and stimulating papers have studied the convergence properties of nonparametric regression estimation based on neural network approximation of the regression function $f_0$
\citep{
bauer2019deep, schmidt2019deep, schmidt2020nonparametric,chen2019nonparametric, kohler2019estimation, nakada2020adaptive,  farrell2021deep}.
These works show that deep neural network regression can achieve the optimal-minimax rate
established by \cite{stone1982optimal} under certain conditions.
However,  the convergence rate can be extremely slow when the dimensionality $d$ of the predictor $X$ is high. Therefore, nonparametric regression using deep neural networks cannot escape the well-known
problem of {\it curse of dimensionality} in high-dimensions without any conditions on the underlying model. There has been much effort devoted to deriving better convergence rates under certain assumptions that mitigate the curse of dimensionality. There are two main types of assumptions in the existing literature:  structural assumptions on the target function $f_0$  \citep{schmidt2020nonparametric,bauer2019deep,kohler2019estimation}
and distributional assumptions on the input $X$  \citep{schmidt2019deep,chen2019nonparametric,nakada2020adaptive}.
Under either of these assumptions, the convergence rate $C_dn^{-2\beta/(2\beta+d)}$  could be improved to {\color{black}$C_{d,d^*} n^{-2\beta/(2\beta+d^*)}$}
for some $d^*\ll d$, where {\color{black}$C_{d,d^*}$
is a constant depending on $(d^*,d)$}
and $d^*$ is
the intrinsic dimension of $f_0$ or
the intrinsic dimension of the support of the predictor.
We will provide a detailed comparison between our results and the
existing results in Section \ref{review}.
	

In this paper, we study the properties of  nonparametric least squares regression using deep neural networks. Our main contributions are as follows:

\begin{enumerate}[(i)]
	{\color{black}
\item  We derive a novel approximation error bound for the H\"older smooth functions with
smoothness index $\beta > 0$ using ReLU activated neural networks. Our work
builds on the results of \cite{shen2019deep} and  \cite{lu2020deep}.  \cite{shen2019deep} derived approximation error bound with prefactor depending on $d$  polynomially for H\"older continuous functions (with smoothness index $\beta\in(0,1]$). \cite{lu2020deep} derived approximation error bound explicitly in network depth and width for higher-order smooth functions (with smoothness index $\beta\ge1$ being positive integer) but with prefactor depending on $d$ exponentially. For $\beta > 1$,  the prefactor of our error bound is significantly improved in the sense that it depends on $d$ polynomially instead of exponentially. This approximation result is of independent interest and may be useful in other problems.
}

\item
We alleviate the curse of dimensionality by assuming that $X$ is supported on an approximate
low-dimensional manifold.
Under such an approximate low-dimensional manifold support assumption, we show that the rate of convergence $O(n^{-2\beta/(2\beta+d)})$ can be improved to  $O(n^{-2\beta/(2\beta+d_\mathcal{M}\log(d))})$, 
where $d_\mathcal{M}$ is the intrinsic dimension of the low-dimensional manifold and $\beta>0$ is the order of the H\"older-smoothness of  $f_0$.
Moreover, under the exact manifold support assumption, we established a 
result that achieves the optimal rate $O(n^{-2\beta/(2\beta+d_\mathcal{M})})$ (up to a logarithmic factor)
with a prefactor only depending linearly on $d$.
{\color{black} We also consider a low Minkowski dimension assumption as in \cite{nakada2020adaptive} and derive an error bound that alleviates the curse of dimensionality with different network architectures and using a different proof technique. }

\item We  derive explicitly how the error bounds are determined by the neural network parameters,
including the width,  the depth and the size of the network. We propose a notion of network relative efficiency between two types of neural networks, defined as the ratio of the logarithms of the network sizes needed to achieve the optimal convergence rate.
This provides a quantitative measure for evaluating the relative merits of network structures.
We quantitatively demonstrate that deep networks have advantages over shallow networks in the sense that they achieve the same error bound with a smaller network size.
			
	\end{enumerate}

The remainder of the paper is organized  as follows.
In Section \ref{sec2} we describe the setup of the problem and the class of ReLU activated feedforward neural networks used in estimating the regression function.
 In Section \ref{sec2new} we present a basic inequality for the excess risk in terms of the stochastic and approximation errors and describe our approach to the analysis of these errors.
 We also establish a novel approximation error bound for the H\"older smooth functions with smoothness index $\beta > 0$  using ReLU activated neural networks,
 In Section \ref{sec3} we provide sufficient conditions under which the neural regression estimator possesses the basic consistency property,
establish non-asymptotic error bounds for the neural regression estimator using deep feedforward neural networks. In Section \ref{efficiency} we  present the results on how the error bounds depend on the network structures and propose a notion of network relative efficiency between two types of neural networks, defined as the ratio of the logarithms of the network sizes needed to achieve the optimal convergence rate. This can be used as a quantitative measure for evaluating the relative merits of different network structures.
 In Section \ref{sec4} we show that the neural regression estimator can circumvent the curse of dimensionality if the data distribution is supported on an (approximate) low-dimensional manifold
 or a set with a low Minkowski dimension.
Detailed comparison between our results and the related works
 are presented in section \ref{review}.
Concluding remarks are given in section \ref{sec5}.
	
\section{Preliminaries}
	\label{sec2}
In this section, we present the basic setup of the nonparametric regression problem and define
the excess risk and the prediction error for which we wish to establish the non-asymptotic error bounds. We also describe the structure of feedforward neural networks to be used in the estimation of the regression function.

\subsection{Least squares estimation}
A basic paradigm for estimating $f_0$ is to minimize the mean squared error or the $L_2$ risk.
For any (random) function $f$, let $Z\equiv (X,Y)$ be a random vector 
independent of  $f$. The $L_2$ risk is defined by $L(f)=\mathbb{E}_{Z}\vert Y-f(X)\vert^2$. At the population level, the least-squares estimation is to find a measurable function
$f^*: \mathbb{R}^d\to \mathbb{R}$ satisfying
	$$f^* :=\arg\min_{f} L(f) =\arg\min_{f}\mathbb{E}_{Z}\vert Y-f(X)\vert^2.$$
Under the assumption that $\mathbb{E}(\eta|X)=0$,  the underlying regression function $f_0$  is the optimal solution $f^*$ on $\mathcal{X}$.	
However, in  applications, the distribution of $(X,Y)$ is typically unknown and only a random sample $S \equiv \{(X_i,Y_i)\}_{i=1}^n$ is available.
Let
\begin{equation}
\label{er1}
L_n(f)=\sum_{i=1}^{n}\vert Y_i-f(X_i)\vert^2/n
\end{equation}
be  the empirical risk of $f$ on the sample $S$. 	
Based on the observed random sample, our primary goal is to construct an estimators of $f_0$ within a certain class of functions $\mathcal{F}_n$ by minimizing the empirical risk.
Such an estimator is called the empirical risk minimizer (ERM), defined by
	\begin{equation}\label{erm}
		\hat{f}_n\in\arg\min_{f\in\mathcal{F}_n}L_n(f).
	\end{equation}
Throughout the paper, we choose $\mathcal{F}_n$  to be a function class consisting of
feedforward neural networks.
For any estimator $\hat{f}_n$,  we evaluate its quality via its  \textit{excess risk},
defined as the difference between the $L_2$
 risks of $\hat{f}_n$ and $f_0$,
\begin{align*}
L(\hat{f}_n)-L(f_0) = \mathbb{E}_{Z}\vert Y- \hat{f}_n(X)\vert^2-\mathbb{E}_{Z}\vert Y- f_0(X)\vert^2. \nonumber
\end{align*}
Because of the simple form of the least squares loss, the excess risk can be simply expressed as
\[
\Vert\hat{f}_n-f_0\Vert^2_{L^2(\nu)}=\mathbb{E}_{X}\vert \hat{f}_n(X)-f_0(X)\vert^2,
\]
where $\nu$ denotes the marginal distribution of  $X$.
A good estimator  $\hat{f}_n$ should have a small excess risk
$\Vert\hat{f}_n-f_0\Vert^2_{L^2(\nu)}.$
Thereafter,  we focus on deriving the non-asymptotic upper bounds of the excess risk $\Vert\hat{f}_n-f_0\Vert^2_{L^2(\nu)}$ and the prediction error $\mathbb{E}_S\Vert\hat{f}_n-f_0\Vert^2_{L^2(\nu)}$.

\subsection{ReLU feedforward neural networks}
In recent years,  deep neural network modeling has achieved impressive successes in many applications.
Also, neural network functions have proven to be
an effective approach for approximating high-dimensional functions.
We consider regression function estimators based on the feedforward neural networks with rectified linear unit (ReLU) activation function.
Specifically, we set the function class $\mathcal{F}_n$ 
 to be $\mathcal{F}_{\mathcal{D},\mathcal{W}, \mathcal{U},\mathcal{S},\mathcal{B}}$, a class of  feedforward neural networks $f_\phi: \mathbb{R}^d \to \mathbb{R} $ with parameter $\phi$, depth $\mathcal{D}$, width $\mathcal{W}$, size $\mathcal{S}$, number of  neurons $\mathcal{U}$ and $f_{\phi}$ satisfying $\Vert f_\phi\Vert_\infty\leq\mathcal{B}$ for some $0 <\mathcal{B} < \infty$, where
	$\Vert f \Vert_\infty$ is the sup-norm of a function $f$.
	Note that the network parameters may depend on the sample size $n$, but  the dependence is omitted in the notation for simplicity.
A brief description of the feedforward neural networks are given below.
	
We begin with the multi-layer perceptron (MLP), an important and widely used subclass of feedforward neural networks in practice. The architecture of a MLP can be expressed as a composition of a series of functions
\[
f_\phi(x)=\mathcal{L}_\mathcal{D}\circ\sigma\circ\mathcal{L}_{\mathcal{D}-1}
\circ\sigma\circ\cdots\circ\sigma\circ\mathcal{L}_{1}\circ\sigma\circ\mathcal{L}_0(x),\  x\in \mathbb{R}^{p_0},
\]
where $p_0=d$ and $\sigma(x)=\max(0, x)$ is the rectified linear unit (ReLU) activation function (defined for each component of $x$ if $x$ is a vector) and
	$\mathcal{L}_{i}(x)=W_ix+b_i, i=0,1,\ldots,\mathcal{D},$
where $W_i\in\mathbb{R}^{p_{i+1}\times p_i}$ is a weight matrix,  $p_i$ is the width (the number of neurons or computational units) of the $i$-th layer, and $b_i\in\mathbb{R}^{p_{i+1}}$ is the bias vector in the $i$-th linear transformation $\mathcal{L}_i$.
The input data consisting of predictor values $X$ is the first layer and the output
is the last layer. Such a network $f_\phi$ has $\mathcal{D}$ hidden layers and $(\mathcal{D}+2)$ layers in total. We use a $(\mathcal{D}+2)$-vector $(p_0,p_1,\ldots,p_\mathcal{D},p_{\mathcal{D}+1})^\top$ to describe the width of each layer; particularly,   $p_0=d$ is the dimension of the input $X$ and $p_{\mathcal{D}+1}=1$ is the dimension of the response $Y$ in model (\ref{model}). The width $\mathcal{W}$ is defined as the maximum width of hidden layers, i.e.,
$\mathcal{W}=\max\{p_1,...,p_\mathcal{D}\}$; 	the size $\mathcal{S}$ is defined as the total number of parameters in the network $f_\phi$, i.e., $\mathcal{S}=\sum_{i=0}^\mathcal{D}\{p_{i+1}\times(p_i+1)\}$; 	
the number of neurons $\mathcal{U}$ is defined as the number of computational units in hidden layers, i.e., $\mathcal{U}=\sum_{i=1}^\mathcal{D} p_i$.  Note that the neurons in consecutive layers of a MLP are connected to each other via linear transformation matrices $W_i$,  $i=0,1,\ldots,\mathcal{D}$.   In other words, an MLP is fully connected between consecutive layers and has no other connections. For an MLP  class $\mathcal{F}_{\mathcal{D},\mathcal{U},\mathcal{W},\mathcal{S},\mathcal{B}}$,
its parameters satisfy the simple relationship
 \[
\max\{\mathcal{W},\mathcal{D}\}\leq\mathcal{S}\leq\mathcal{W}(d+1)
+(\mathcal{W}^2+\mathcal{W})(\mathcal{D}-1)+\mathcal{W}+1
=O(\mathcal{W}^2\mathcal{D}).
\]
The network parameters can depend on the sample size $n$, that is,
$\mathcal{S}=\mathcal{S}_n$, $\mathcal{D}=\mathcal{D}_n$,
$\mathcal{W}=\mathcal{W}_n$, and $\mathcal{B}=\mathcal{B}_n$.
This makes it possible to approximate the target regression function by neural networks as $n$ increases.
For notational simplicity, we omit the subscript below.
The approximation and excess error rates will be determined in part
by how these network parameters depend on $n$.

Different from multilayer perceptrons, a general feedforward neural network may not be fully connected. 	For such a network, 
each neuron in layer $i$ may be connected to only a small subset of neurons in layer $i+1$.
The total number of parameters $\mathcal{S}$ is reduced and the computational cost required to evaluate the network will also be reduced.
	
Though our discussion focuses on multi-layer perceptrons due to their simplicity, our theoretical results are valid for general feedforward neural networks. Moreover, our results for ReLU networks can be extended to networks with piecewise-linear activation functions without further difficulty,
based on the  approximation results  \citep{yarotsky2017error} and the VC-dimension
bounds \citep{bartlett2019nearly} for piecewise linear neural networks.

\section{Basic error analysis}
\label{sec2new}
In this section, we present a basic inequality for the excess risk in terms of the stochastic and approximation errors and describe our approach to the analysis of these errors.
	\subsection{A basic inequality}
	\label{sec2.1}
	To begin with, 	we  give a basic upper bound on the excess risk of the empirical risk minimizer. For a general loss function $L$ and any estimator $\hat{f}_n$ belonging to
a function class $\mathcal{F}_n$, its excess risk can be decomposed as \citep{mohri2018foundations}:
	$$L(\hat{f}_n)-L(f_0)=\left\{L(\hat{f}_n)-\inf_{f\in\mathcal{F}_n} L(f)\right\}+ \left\{\inf_{f\in\mathcal{F}_n}L(f)-L(f_0) \right\}.$$
	The first term of the right hand side is the {\it stochastic error}, and the second term is  the {\it approximation error}.
	The stochastic error depends on the estimator $\hat{f}_n$, which measures the difference of the error of $\hat{f}_n$ and the best one in  $\mathcal{F}_n$. The approximation error depends on the function class $\mathcal{F}_n$ and the target $f_0$, which measures how well the function $f_0$ can be approximated using $\mathcal{F}_n$ with respect to the loss $L$.
	
For least squares estimation, the loss function $L$ is the $L_2$ loss and $\hat f_n$ is the ERM  
defined in (\ref{erm}).
We firstly establish an upper bound on the excess risk of $\hat f_n$ with least squares loss.
	

	\begin{lemma}\label{lemma0}
For any random sample $S=\{(X_i, Y_i)\}_{i=1}^n$,
the excess risk of ERM satisfies
		\begin{align*}		\mathbb{E}_{S} [\|\hat{f}_n-f_0\|_{L^2(\nu)}^2]&=
\mathbb{E}_{S} [L(\hat{f}_n)-L(f_0)]\\
			&\leq \mathbb{E}_{S}[ L(f_0)-2L_{n}(\hat{f}_n)+L(\hat{f}_n)]+2\inf_{f\in\mathcal{F}_n} \|f-f_0\|_{L^2(\nu)}^2.
		\end{align*}
	\end{lemma}	
By Lemma \ref{lemma0},
the excess risk of ERM  is bounded above by the sum of two terms: the stochastic error bound $\mathbb{E}_{S}[ L(f_0)-2L_{n}(\hat{f}_n)+L(\hat{f}_n)]$ and the approximation error $\inf_{f\in\mathcal{F}_n} \Vert f-f_0\Vert^2_{L^2(\nu)}$.
The first term $\mathbb{E}_{S}[ L(f_0)-2L_{n}(\hat{f}_n)+L(\hat{f}_n)]$ can be bounded by the complexity of $\mathcal{F}_n$ using the empirical process theory \citep{vw1996, anthony1999, bartlett2019nearly}. The second term $\inf_{f\in\mathcal{F}_n} \Vert f-f_0\Vert^2_{L^2(\nu)}$ measures the approximation error of the function class $\mathcal{F}_n$ to $f_0$.
The approximation of high-dimensional functions using neural networks has been studied by many authors, some recent works include \citet{yarotsky2017error, yarotsky2018optimal,shen2019nonlinear, shen2019deep,lu2020deep,shen2021optimal}, among others.
	
\subsection{Stochastic error}
	\label{sec2.2}
In this subsection, we focus on the stochastic error of ERM implemented using the feedforward neural networks and establish an upper bound on the prediction error,  or the expected excess risk.  For the  least-squares estimator of neural networks nonparametric regression, oracle inequalities for a bounded response variable were studied by \cite{gyorfi2006distribution} and \cite{farrell2021deep}. Without the boundedness assumption on $Y$, \cite{schmidt2020nonparametric,bauer2019deep} derived the oracle inequality for a sub-Gaussian $Y$. We consider 
a sub-exponentially distributed $Y$.

\begin{assumption}
		\label{A1}
The response variable $Y$ is sub-exponentially distributed, i.e., there exists a constant $\sigma_Y>0$ such that $\mathbb{E}\exp(\sigma_Y\vert Y\vert)<\infty$.
	\end{assumption}

For a class $\mathcal{F}$ of functions: $\mathcal{X}\to \mathbb{R}$,
 its pseudo dimension, denoted by $\text{Pdim}(\mathcal{F}),$  is the largest integer $m$ for which there exists $(x_1,\ldots,x_m,y_1,\ldots,y_m)\in\mathcal{X}^m\times\mathbb{R}^m$ such that for any $(b_1,\ldots,b_m)\in\{0,1\}^m$ there exists $f\in\mathcal{F}$ such that $\forall i:f(x_i)>y_i\iff b_i=1$ \citep{anthony1999, bartlett2019nearly}.
 For a class of real-valued functions generated by neural networks, pseudo dimension is a natural measure of its complexity. In particular, if $\mathcal{F}$ is the class of functions generated by a neural network with a fixed architecture and fixed activation functions, we have $\text{Pdim}(\mathcal{F})=\text{VCdim}(\mathcal{F})$ (Theorem 14.1 in \cite{anthony1999}) where $\text{VCdim}(\mathcal{F})$ is the VC dimension of $\mathcal{F}$.
 In our results, we require the sample size $n$ to be greater than the pseudo dimension of the class of neural networks considered.
	

For a given sequence $x=(x_1,\ldots,x_n)\in\mathcal{X}^n,$
let  $\mathcal{F}_n|_x=\{(f(x_1),\ldots,f(x_n):f\in\mathcal{F}_n\}$ be the  subset of $\mathbb{R}^{n}$. For a positive number $\delta$, let $\mathcal{N}(\delta,\Vert\cdot\Vert_\infty,\mathcal{F}_n|_x)$ be the covering number of $\mathcal{F}_n|_x$ under the norm $\Vert\cdot\Vert_\infty$ with radius $\delta$.
Define the uniform covering number
$\mathcal{N}_n(\delta,\Vert\cdot\Vert_\infty,\mathcal{F}_n)$ to be the maximum
	over all $x\in\mathcal{X}$ of the covering number $\mathcal{N}(\delta,\Vert\cdot\Vert_\infty,\mathcal{F}_n|_x)$, i.e.,
\begin{equation}
\label{ucover}
\mathcal{N}_n(\delta,\Vert\cdot\Vert_\infty,\mathcal{F}_n)=
\max\{\mathcal{N}(\delta,\Vert\cdot\Vert_\infty,\mathcal{F}_n|_x):x\in\mathcal{X}\}.
\end{equation}

\begin{lemma}\label{lemma1}
		Consider the $d$-variate nonparametric regression model in (\ref{model}) with an unknown regression function $f_0$. Let $\mathcal{F}_n=\mathcal{F}_{\mathcal{D},\mathcal{W},\mathcal{U},\mathcal{S},\mathcal{B}}$ be the class of feedforward neural networks with a continuous piecewise-linear activation function with finitely many inflection points
and $\hat{f}_n \in\arg\min_{f\in\mathcal{F}_n}L_n(f) $  be the empirical risk minimizer over $\mathcal{F}_n$. Assume that Assumption \ref{A1} holds and $ \Vert f_0\Vert_\infty\leq \mathcal{B}$ for $\mathcal{B}\geq 1$.
Then,  for $n \ge \text{Pdim}(\mathcal{F}_n)/2$,
\begin{equation} \label{bound5a}
\mathbb{E}_{S}[ L(f_0)-2L_{n}(\hat{f}_n)+L(\hat{f}_n)]  \leq c_0\mathcal{B}^4 (\log n)^4\,
\frac{1}{n} \log\mathcal{N}_{2n}(n^{-1},\Vert \cdot\Vert_\infty,\mathcal{F}_n),
\end{equation}
where $c_0>0$ is a constant independent of $d$, $n$, $\mathcal{B}$, $\mathcal{D}$, $\mathcal{W}$ and $\mathcal{S}$, and
\begin{equation} \label{oracle}
		\mathbb{E}\Vert \hat{f}_n-f_0\Vert^2_{L^2(\nu)}\leq C_0\mathcal{B}^5(\log n)^5\, \frac{1}{n}\mathcal{S}\mathcal{D}\log(\mathcal{S})+ 2\inf_{f\in\mathcal{F}_n}\Vert f-f_0\Vert^2_{L^2(\nu)},
		\end{equation}
where $C_0>0$ is a constant independent of $d$, $n$, $\mathcal{B}$, $\mathcal{D}$,  $\mathcal{W}$ and $\mathcal{S}$.
	\end{lemma}
The stochastic error is bounded by a term determined by the metric entropy of $\mathcal{F}_n$
in (\ref{bound5a}),
which is measured by the covering number of $\mathcal{F}_n$.
	To obtain (\ref{oracle}), we further bound the covering number of $\mathcal{F}_n$ by its pseudo dimension (VC dimension).
	Based on \cite{bartlett2019nearly}, the pseudo dimension (VC dimension) of $\mathcal{F}_n$ with piecewise-linear activation function can be further contained and represented by its parameters $\mathcal{D}$ and $\mathcal{S}$, i.e.,  ${\rm Pdim}(\mathcal{F}_n)=O(\mathcal{S}\mathcal{D}\log(\mathcal{S}))$.
This leads to the upper bound for the prediction error by the sum of the stochastic error and  the approximation error of $\mathcal{F}_n$ to $f_0$ in (\ref{oracle}).

Results similar to Lemma \ref{lemma1} with slightly different constants have been obtained for a bounded $Y$ in \citet{gyorfi2006distribution} and a sub-Gaussian $Y$ in \citet{bauer2019deep}
and \citet{schmidt2020nonparametric}.

\subsection{Approximation error}
\label{sec2.3}
The 
approximation error depends on $\mathcal{F}_n=\mathcal{F}_{\mathcal{D},\mathcal{W},\mathcal{U},\mathcal{S},\mathcal{B}}
$ through its parameters and is related to the smoothness of $f_0$. The existing works on
approximation posit different smoothness assumptions on $f_0$. For example,  \citet{bauer2019deep} assume that $f_0$ is $\beta$-H\"older smooth with $\beta \ge 1$, i.e.,  all partial derivatives of $f_0$ up to order $\lfloor\beta\rfloor$ exist and the partial derivatives of order $\lfloor\beta\rfloor$ are $\beta-\lfloor\beta\rfloor$ H\"older continuous.
\citet{farrell2021deep} requires that $f_0$ lies in a Sobolev ball with smoothness $\beta\in \mathbb{N}^+$, i.e., $f_0(x)\in\mathcal{W}^{\beta,\infty}([-1,1]^d)$.
Approximation theories on Korobov spaces (\cite{mohri2018foundations}), Besov spaces \citep{suzuki2018adaptivity} or function space with $f_0\in C^\beta[0,1]^d$ with integer $\beta\ge1$
can be found in \cite{liang2016deep}, \cite{lu2017expressive}, \cite{yarotsky2017error} and \cite{lu2020deep}.

Here, we assume that $f_0$ is a $\beta$-H\"older smooth function as stated in Assumption \ref{A2} below. We aim to develop an approximation theory by utilizing the smoothness of $f_0$ and obtain an explicit approximation error bound in terms of the network depth and width with an improved prefactor compared to previous results.

Let $\beta=s+r >0$,  $r\in(0,1]$ and $s=\lfloor\beta\rfloor\in\mathbb{N}_0$, where $\lfloor\beta\rfloor$ denotes the largest integer strictly smaller than $\beta$ and $\mathbb{N}_0$ denotes the set of non-negative integers. For a finite constant $B_0 >0$, the H\"older class of functions $\mathcal{H}^\beta([0,1]^d,B_0)$  is defined as

\begin{align}
\label{Hclass}
&\mathcal{H}^\beta([0, 1]^d,B_0)  \\
&\ =
\Big\{f: [0,1]^d\to\mathbb{R},\max_{\Vert\alpha\Vert_1\le s}\Vert\partial^\alpha f\Vert_\infty\le B_0,
\max_{\Vert\alpha\Vert_1=s} \sup_{x\not=y}\frac{\vert\partial^\alpha f(x)-\partial^\alpha f(y)\vert}{\Vert x-y\Vert_2^r}\le B_0 \Big\}, \nonumber
\end{align}
where $\partial^\alpha=\partial^{\alpha_1}\cdots\partial^{\alpha_d}$ with $\alpha=(\alpha_1,\ldots,\alpha_d)^\top\in\mathbb{N}_0^d$ and $\Vert\alpha\Vert_1=\sum_{i=1}^d\alpha_i$.

\begin{assumption}[H\"older smoothness]\label{A2}
The target function $f_0$ belongs to the H\"older class $\mathcal{H}^\beta([0,1]^d,B_0)$ defined in (\ref{Hclass}) for a given $\beta > 0$ and a finite constant $B_0 > 0$ .
\end{assumption}

Under Assumption \ref{A2}, all partial derivatives of $f_0$ up to the
$ \lfloor \beta\rfloor$-th order exist.
When $\beta\in(0,1)$, $f_0$ is a H\"older continuous function with order $\beta$ and H\"older constant $B_0$;
when $\beta=1$, $f_0$ is a Lipschitz function with Lipschitz constant $B_0$;
when $\beta > 1$, $f_0$ belongs to the $C^s$ class (class of functions whose $s$-th partial derivatives exist and are bounded) with $s=\lfloor \beta\rfloor.$

In this work, the function class $\mathcal{F}_n$ consists of the feedforward neural networks with the ReLU activation function.
An important result on deep neural network approximation proved by \cite{yarotsky2017error} is the following:
for any $\varepsilon\in(0,1)$, any $d,\beta$,  and any $f_0$ in the Sobolev ball $\mathcal{W}^{\beta,\infty}([0,1]^d)$ with $\beta > 0$,
there exists a ReLU network $\hat{f}$ with depth $\mathcal{D}$ at most $c\{\log(1/\varepsilon)+1\}$, size $\mathcal{S}$ and number of neurons $\mathcal{U}$ at most $c\varepsilon^{-d/\beta}\{\log(1/\varepsilon)+1\}$ such that
$\Vert \hat{f}-f_0\Vert_\infty\equiv \max_{x\in[0,1]^d} \vert \hat{f}(x)-f_0(x)\vert\leq \varepsilon$, where $c$ is some constant depending on $d$ and $\beta$.
In particular, it is required that the constant $c=O(2^d)$, an exponential rate of $d$,  due to the technicality in the proof.  	
The main idea of \citet{yarotsky2017error} is to show that, small neural networks can approximate polynomials well locally, and stacked neural networks (by $2^d$ small sub-networks) can further approximate smooth function by approximating its Taylor expansions.
	 \citet{yarotsky2018optimal} derived the optimal rate of approximation for continuous functions   by deep ReLU networks in terms of the network size $\mathcal{S}$ and the modulus of continuity of $f_0$. It was shown that $\inf_{f\in\mathcal{F}_n}\Vert f-f_0\Vert_\infty\leq c_1\omega_{f_0}(c_2\mathcal{S}^{-p/d})$ for some $p\in[1,2]$ and some constants $c_1,c_2$ possibly depending on $d,p$ but not $\mathcal{S}, f_0$.
The upper bound holds for any $p\in(1,2]$ if the network $\mathcal{F}_n=\mathcal{F}_{\mathcal{D},\mathcal{W},\mathcal{U},\mathcal{S},\mathcal{B}}$ satisfies $\mathcal{D}\geq c_3\mathcal{S}^{p-1}/\log(\mathcal{S})$ for some constant $c_3$ possibly depending on $p$ and $d$.
\citet{shen2021optimal} established the optimal rate of approximation for H\"older continuous functions  by deep ReLU networks in terms both width and depth. They showed by construction that deep ReLU networks with width $\mathcal{W}= O((\max \{d\lfloor N^{1 / d}\rfloor, N+2 \}))$ and depth $\mathcal{D} = O(L)$ can approximate a H\"older continuous function on $[0,1]^{d}$ with an approximation rate $O(B_0 \sqrt{d}(N^{2} L^{2} \log N)^{-\beta / d})$, where $\beta \in(0,1]$ and $B_0>0$ are  H\"older order and constant, respectively.

Several recent studies have considered approximation properties of deep neural networks
\citep{chen2019efficient, nakada2020adaptive, schmidt2019deep,schmidt2020nonparametric}.
These studies used a construction similar to that of \cite{yarotsky2017error}.
A common feature of these results is that, the prefactor of the approximation error is of the order $O(a^d)$ for some $a\geq2$ and the size $\mathcal{S}$ or the width $\mathcal{W}$ of the network grows at least exponentially in $d.$
Unfortunately, a prefactor of the order $O(a^d)$ with $a \ge 2$ can be very large even for a moderate $d$, which severely deteriorates the quality of the error bound.
For example, for a typical genomic dataset,
 the dimensionality $d =20,531$ and the sample size $n=801$ \citep{2013Tcga},
 which leads to a prohibitively large prefactor.

Next, we present a new ReLU network approximation result for H\"older smooth functions in $\mathcal{H}^\beta([0,1]^d,B_0)$ with a prefactor in the error bound only depending on the
dimension $d$ polynomially, i.e.,   $d^{\lfloor\beta\rfloor+(\beta\vee1)/2}$.

\begin{theorem} \label{thm_apx}
Assume  that $f\in\mathcal{H}^\beta([0,1]^d,B_0)$ with $\beta=s+r$, $s\in\mathbb{N}_0$ and $r\in(0,1]$. For any $M,N\in\mathbb{N}^+$, there exists a function $\phi_0$  implemented by a ReLU network with width $\mathcal{W}=38(\lfloor\beta\rfloor+1)^2d^{\lfloor\beta\rfloor+1}N\lceil\log_2(8N)\rceil$ and depth $\mathcal{D}=21(\lfloor\beta\rfloor+1)^2M\lceil\log_2(8M)\rceil$ such that
	$$\vert f(x)-\phi_0(x)\vert\leq 18B_0(\lfloor\beta\rfloor+1)^2d^{\lfloor\beta\rfloor+(\beta\vee1)/2}(NM)^{-2\beta/d},$$
	for all $x\in[0,1]^d\backslash\Omega([0,1]^d,K,\delta)$, where
	$a\vee b:=\max\{a,b\}$, $\lceil a\rceil$ denotes the smallest integer no less than $a$, and
	$$\Omega([0,1]^d,K,\delta)=\bigcup_{i=1}^d\{x=[x_1,x_2,\ldots,x_d]^\top:x_i\in\bigcup_{k=1}^{K-1}(k/K-\delta,k/K)\},$$
with $K=\lceil (MN)^{2/d}\rceil$ and $\delta$ an arbitrary number in $(0,{1}/(3K)]$.
\end{theorem}

Theorem \ref{thm_apx} is inspired by and builds on the  work of \cite{shen2019deep} and \cite{lu2020deep}.
Similar to the results of \cite{shen2019deep} and \cite{lu2020deep}, the  approximation error
bound in Theorem \ref{thm_apx} has the optimal approximation rate $(NM)^{-2\beta/d}$.
This error bound is non-asymptotic in the sense that it is valid for arbitrary network width and depth specified by $N$ and $M$.  The error bound is also explicit since no unknown or undefined parameters are involved.
Moreover, our error bound is given in terms of the network width and depth, which is more informative than the bounds just in terms of the network size as in many existing works.

However,  the prefactor in the approximation error bound and the network width in
Theorem \ref{thm_apx} are different from those in the result of \cite{lu2020deep},
who showed that,  for a positive integer $\beta$, and suppose that the network width and depth are chosen to be $16\beta^{d+1}(N+2)\log_2(8N)$ and $18\beta^{2}(M+2)\log_2(4M),$ respectively, the approximation error bound  is of the form $84(\beta+1)^d8^\beta(NM)^{-2\beta/d}$.
The prefactor in this bound depends on $d$ 
exponentially through the term $(\beta+1)^d 8^{\beta}$.
In comparison, the prefactor in the error bound
in Theorem \ref{thm_apx} depends on $d$ 
polynomially through $(\lfloor\beta\rfloor+1)^2d^{\lfloor\beta\rfloor+(\beta\vee1)/2}.$
This is a significant improvement
for a large $d$ with a moderate $\beta,$ which is a probable situation in nonparametric regression.
 Even in the unlikely case where $\beta= O(d)$ is a large number, our prefactor is still comparable with $O((\beta+1)^d 8^{\beta}).$



The basic idea of our proof follows that of  \cite{lu2020deep}:
we approximate a H\"older smooth function $f$ using Taylor expansion locally over
a discretization of $[0, 1]^d,$ however, we have a more careful control of the number of the partial derivatives. More specifically,
our proof consists of three steps: (a) we first construct a network $\psi$  that 
discretizes  $[0,1]^d$; (b)  we construct a second network $\phi_{\alpha}$ to approximate the Taylor coefficient; (c) We construct a third network $P_{\alpha}(x)$ to approximate the polynomial $x^{\alpha}$. Putting all these together, we use
$$\phi(x)=\sum_{\Vert\alpha\Vert_1\le s}\phi_\times\Big( \frac{\phi_\alpha(x)}{\alpha!}, P_\alpha(x-\psi(x))\Big)$$ to approximate $f$,
where $\phi_{\times}(\cdot, \cdot)$ is a network function  approximating the product function of two scalar inputs.

To use the information of higher order smoothness, the existing results such as
\citet{yarotsky2017error} and \cite{lu2020deep}, are also based on the idea of approximating the Taylor expansion of the target function locally on a discretized hyper cube. Two key components of the technique used in the proof affects the prefactor of the approximation error:  (a) how the hyper cube is discretized and the target function is locally approximated;  (b) how the number of partial derivatives is upper bounded. We use the method of discretization and local approximation in  \cite{lu2020deep}, which avoids the $2^d$ prefactor appeared in \citet{yarotsky2017error} and \citet{schmidt2020nonparametric}. At the same time, we changed the way of bounding the number of partial derivatives, which leads to a $O(d^\beta)$ prefactor instead of $O(8^{\beta}(\beta+1)^d)$ in  \cite{lu2020deep} and $O((2e)^d (\beta+1)^d)$ in
Theorem 5 of \citet{schmidt2020nonparametric}.
The $d^\beta$ prefactor is clearly an improvement over  $(\beta+1)^d$ when $d$ is large and $\beta$ is moderate.

Based on Theorem \ref{thm_apx}, we can establish the approximation error bounds under the $L^p({\nu})$ norm for $p\in(0,\infty)$ with an absolutely continuous $\nu$ (with respect to the Lebesgue measure on $\mathbb{R}^d$ ).
For the approximation result under  the $L^\infty([0,1]^d)$ norm, we have the following
corollary of Theorem \ref{thm_apx}.

\begin{corollary}\label{thm_apx1}
	Assume that $f\in\mathcal{H}^\beta([0,1]^d, B_0)$ with $\beta=s+r$, $s\in\mathbb{N}_0$ and $r\in(0,1]$. For any $M,N\in\mathbb{N}^+$, there exists a function $\phi$ implemented by a ReLU network  with width $\mathcal{W}=38(\lfloor\beta\rfloor+1)^23^dd^{\lfloor\beta\rfloor+1}N\lceil\log_2(8N)\rceil$ and depth $\mathcal{D}=21(\lfloor\beta\rfloor+1)^2M\lceil \log_2(8M)\rceil+2d$ such that
	$$\vert f(x)-\phi(x) \vert\leq 19B_0(\lfloor\beta\rfloor+1)^2
d^{\lfloor\beta\rfloor+(\beta\vee1)/2}(NM)^{-2\beta/d},\ x \in [0,1]^d.$$
\end{corollary}

The approximation error under  $L^\infty([0,1]^d)$
is the same as that of Theorem \ref{thm_apx}, at the price that the network width
should be as large as $3^d$ times of that in Theorem \ref{thm_apx}.

Lastly, we note that, by Proposition 1 of \cite{yarotsky2017error}, in terms of the computational power and complexity of a neural network, there is no substantial difference in using the ReLU activation function and  other piece-wise linear activation functions with finitely many
inflection points.
To elaborate, let $\zeta: \mathbb{R}\to\mathbb{R}$ be any continuous piece-wise linear function with $M$
inflection points
($1\leq M<\infty$). If a network $f_\zeta$ is activated by $\zeta$, of depth $\mathcal{D}$, size $\mathcal{S}$ and the number of neurons $\mathcal{U}$, then there exists a ReLU activated network with depth $\mathcal{D}$, size not more than $(M+1)^2\mathcal{S}$, the number of neurons not more than $(M+1)\mathcal{U}$, that computes the same function as $f_\zeta$. Conversely, let $f_\sigma$ be a ReLU activated network  of depth $\mathcal{D}$, size $\mathcal{S}$ and the number of neurons $\mathcal{U}$, then there exists a network with  activation function $\zeta$, of depth $\mathcal{D}$, size $4\mathcal{S}$ and
the number of neurons $2\mathcal{U}$ that computes the same function $f_\sigma$ on a bounded subset of $\mathbb{R}^d$.
	
\section{Non-asymptotic error bounds} 
\label{sec3}
	
Lemma \ref{lemma1} provides the basis for establishing the consistency and non-asymptotic error bounds. To ensure consistency, the two items on the right hand side of (\ref{oracle}) should vanish as $n\to \infty$. For the non-asymptotic error bound,  the exact rate of convergence will be determined by a trade-off between the stochastic error and the approximation error.
We first state a consistency result and then present the result on the non-asymptotic error bound
of nonparametric regression estimator using neural networks.

\begin{theorem}[Consistency]\label{thm1}
		Under model (\ref{model}), suppose that Assumption \ref{A1} holds, the target function  $f_0$ is  continuous  on $[0,1]^d$, and $ \Vert f_0\Vert_\infty\leq\mathcal{B}$ for some $\mathcal{B}\geq1$, and the function class of feedforward neural networks $\mathcal{F}_n=\mathcal{F}_{\mathcal{D},\mathcal{W},\mathcal{U},\mathcal{S},\mathcal{B}}$ with continuous piecewise-linear activation function  with finitely many
inflection points satisfies
		\begin{equation*}
			\mathcal{S}\to\infty \quad {\rm and} \quad \mathcal{B}^5(\log n)^5\, \frac{1}{n} \mathcal{S}\mathcal{D}\log(\mathcal{S})\to 0, \ \text{ as } n \to \infty.
		\end{equation*}
Then,  the prediction error of the empirical risk minimizer $\hat{f}_n$  is consistent in the sense that
	$$\mathbb{E} \Vert \hat{f}_n-f_0\Vert^2_{L^2(\nu)}\to 0 \ \  {\rm as\ } n\to\infty.$$
\end{theorem}
	Theorem {\ref{thm1}}  is a direct consequence of Lemma \ref{lemma1} and
Theorem 1 on  the approximation
of continuous function by ReLU neural networks in \cite{yarotsky2018optimal}.
The conditions in Theorem \ref{thm1} are sufficient for the consistency of the deep neural regression, and they are relatively mild in terms of the assumptions on the underlying target $f_0$ and the distribution of $Y$.
{\cite{van1996consistency} gave the sufficient and necessary conditions for the consistency of the least squares estimation in nonparametric regression model (\ref{model}) under the assumptions
that $f_0\in\mathcal{F}_n$, the error $\eta$ is symmetric about 0 and it has zero point mass at 0. 
Their results are for the convergence of the empirical error $\Vert \hat{f}_n-f_0\Vert^2_n:=\sum_{i=1}^n\vert \hat{f}_n(X_i)-f_0(X_i)\vert^2/n.$
}

{\begin{theorem}[Non-asymptotic error bound]\label{thm2}
		Under model (\ref{model}), suppose that Assumptions \ref{A1}-\ref{A2} hold, the probability measure of the covariate $\nu$ is absolutely continuous with respect to the Lebesgue measure and $\mathcal{B}\geq \max\{B_0,1\}$. Then, for any $N, M\in\mathbb{N}^+$, the function class of ReLU multi-layer perceptrons $\mathcal{F}_n=\mathcal{F}_{\mathcal{D},\mathcal{W},\mathcal{U},\mathcal{S},\mathcal{B}}$ with width $\mathcal{W}=38(\lfloor\beta\rfloor+1)^2d^{\lfloor\beta\rfloor+1}N\lceil\log_2(8N)\rceil$ and depth $\mathcal{D}=21(\lfloor\beta\rfloor+1)^2M\lceil\log_2(8M)\rceil$, for $n \ge \text{Pdim}(\mathcal{F}_n)/2$,
		the prediction error of the ERM
		$\hat{f}_n$ satisfies
		$$\mathbb{E} \Vert \hat{f}_n-f_0\Vert^2_{L^2(\nu)}\leq C\mathcal{B}^5
		(\log n)^5\, \frac{1}{n}\mathcal{S}\mathcal{D}\log(\mathcal{S}) +324B_0^2(\lfloor\beta\rfloor+1)^4d^{2\lfloor\beta\rfloor+\beta\vee1}(NM)^{-4\beta/d},$$
		where $C>0$ is a constant not depending on $n,d,\mathcal{B},\mathcal{S},\mathcal{D},B_0,\beta,N$ or $M$.
	\end{theorem}
}

Under the assumption that the target function $f_0$ belongs to a H\"older class,  non-asymptotic error bounds can be established. {Similar results have been shown by \citet{bauer2019deep,nakada2020adaptive,schmidt2020nonparametric} and \citet{kohler2021rate}.}
Our error bound is different from the existing ones in the sense that the prefactor of our approximation error depends on $d$ polynomially, instead of
exponentially.

The upper bound of the prediction error 
in Theorem \ref{thm2} is a sum of the upper bound on the stochastic error $C\mathcal{B}^5{\mathcal{S}\mathcal{D}\log(\mathcal{S})(\log n)^5}/{n}$ and  the approximation error $324B_0^2(\lfloor\beta\rfloor+1)^4d^{2\lfloor\beta\rfloor+\beta\vee1}(NM)^{-4\beta/d}$.
Two important aspects worth noting. First, our error bound is non-asymptotic and explicit  in the sense that  no unclearly defined 
constant is involved. The prefactor $324B_0^2(\lfloor\beta\rfloor+1)^4d^{2\lfloor\beta\rfloor+\beta\vee1}$ in the upper bound of approximation error depends on the dimension $d$ polynomially, drastically different from the exponential dependence in existing results. Second, the approximation rate $(NM)^{-4\beta/d}$
is in terms of the width $\mathcal{W}=38(\lfloor\beta\rfloor+1)^2d^{\lfloor\beta\rfloor+1}N\lceil\log_2(8N)\rceil$ and depth $\mathcal{D}=21(\lfloor\beta\rfloor+1)^2M\lceil\log_2(8M)\rceil$, rather than just the size $\mathcal{S}$ of the network.
This provides insights into the relative merits of different the network designs and provides some
qualitative guidance on the network design.

To achieve the best error rate,
we need to balance the trade-off between the stochastic error and the approximation error.
On one hand, the upper bound for the stochastic error $C\mathcal{B}^5{\mathcal{S}\mathcal{D}\log(\mathcal{S})(\log n)^5}/{n}$ increases as the complexity and richness of $\mathcal{F}_{\mathcal{D},\mathcal{W},\mathcal{U},\mathcal{S},\mathcal{B}}$ increase;  larger $\mathcal{D}$, $\mathcal{S}$ and $\mathcal{B}$ lead to a larger upper bound on the stochastic error. On the other hand, the upper bound for the approximation error $324B_0^2(\lfloor\beta\rfloor+1)^4d^{2\lfloor\beta\rfloor+\beta\vee1}(NM)^{-4\beta/d}$  decreases as the size of  $\mathcal{F}_{\mathcal{D},\mathcal{W},\mathcal{U},\mathcal{S},\mathcal{B}}$ increases;   
larger $\mathcal{D}$ and $\mathcal{W}$ lead to smaller upper bound on the approximation error.


In Section \ref{efficiency} we present the specific error bounds for various designs of network structures,  including detailed descriptions of how the prefactors in these bounds depend on the
dimension $d$ of the predictor.
	
\section{Comparing network structures}
\label{efficiency}
Theorem \ref{thm2} provides an explicit expression of how the non-asymptotic error bounds depend on the network parameters, which can be used to quantify the relative efficiency of  networks with different shapes in terms of the network size  needed to achieve the optimal error bound. The calculations given below demonstrate the advantages of deep networks over shallow ones in the sense that deep networks can achieve the same error bound as the shallow networks with a fewer total number of parameters in the network.
We will make this statement quantitatively clear in terms of the notion of relative efficiency between
networks defined below.

\subsection{Relative efficiency of network structures}
Let $\cS_1$ and $\cS_2$ be the sizes of two neural networks
$\cN_1$ and $\cN_2$ needed to achieve the same non-asymptotic error bound as given in Theorem \ref{thm2}. We define the \textit{network relative efficiency} between two networks $\cN_1$ and $\cN_2$ as
\begin{equation}
\label{re}
\text{NRE}(\cN_1, \cN_2) = \frac{\log \cS_2}{\log \cS_1}.
\end{equation}
Here we use the logarithm of the size because the size of the network for achieving the optimal error rate has the form $\cS=[n^{d/(d+2\beta)}]^s$ for some $s > 0$ up to a factor only involving the power of $\log n$, as will be seen below.  Let $r=\text{NRE}(\cN_1, \cN_2)$.
In terms of sample complexity, this definition of relative efficiency implies that, if it takes a sample of size
$n$ for network $\cN_1$ to achieve the optimal error rate, then it will take a sample of size
$n^r$ to achieve the same error rate.

For any multilayer neural network in $\mathcal{F}_{\mathcal{D},\mathcal{W},\mathcal{U},\mathcal{S},\mathcal{B}}$, its parameters naturally satisfy
\begin{equation}
 \label{size}
\max\{\mathcal{W},\mathcal{D}\}\leq
\mathcal{S}\leq
\mathcal{W}(d+1)+(\mathcal{W}^2+\mathcal{W})(\mathcal{D}-1)+\mathcal{W}+1
=O(\mathcal{W}^2\mathcal{D}).
\end{equation}
Corollaries \ref{c1}-\ref{c3} below follow from this relationship and Theorem \ref{thm2}.

\begin{corollary}[Deep with fixed width networks]
	\label{c1}
		Under model (\ref{model}), suppose that Assumptions  \ref{A1}-\ref{A2} hold, $\nu$ is absolutely continuous with respect to the Lebesgue measure, and $\mathcal{B}\geq\max\{1,B_0\}$. Then, for any $N\in\mathbb{N}^+$ and the function class of ReLU multi-layer perceptrons $\mathcal{F}_n=\mathcal{F}_{\mathcal{D},\mathcal{W},\mathcal{U},\mathcal{S},\mathcal{B}}$ with depth $\mathcal{D}$,  width $\mathcal{W}$ and size  $\mathcal{S} $ given by
$
\mathcal{D}=21(\lfloor\beta\rfloor+1)^2 \lceil n^{d/2(d+2\beta)}\log_2(8n^{d/2(d+2\beta)})\rceil,
\mathcal{W}=38(\lfloor\beta\rfloor+1)^2d^{s+1}N\lceil\log_2(8N)\rceil, \
\mathcal{S}=O(n^{d/2(d+2\beta)}\log_2 n),
$
the ERM $\hat{f}_n \in\arg\min_{f\in\mathcal{F}_n}L_n(f) $ satisfies
\begin{align*}
	\mathbb{E} \Vert \hat{f}_n-f_0\Vert^2_{L^2(\nu)}\leq& \Big\{c_1\mathcal{B}^5{\color{black}(\log n)^8}+324B_0^2d^{2\lfloor\beta\rfloor+\beta\vee1}N^{-4\beta/d}\Big\}(\lfloor\beta\rfloor+1)^4n^{-2\beta/(d+2\beta)},\\
	\leq&c_2\mathcal{B}^5N^{-4\beta/d}(\lfloor\beta\rfloor+1)^{4}d^{2\lfloor\beta\rfloor+\beta\vee1}
{\color{black}(\log n)^8}n^{-2\beta/(d+2\beta)},
\end{align*}
for $n \ge \text{Pdim}(\mathcal{F}_n)/2$, where $c_1,c_2>0$ are constants which do not depend on $n,\mathcal{B},B_0,\beta$ or $N$.
	\end{corollary}

Corollary \ref{c1} is a direct consequence of Theorem \ref{thm2}. We note that the prefactor  depends on $d$ at most polynomially.
	
	
\begin{corollary}[Wide with fixed depth networks]
\label{c2}
		Under model (\ref{model}), suppose that Assumptions \ref{A1}-\ref{A2} hold, $\nu$ is absolutely continuous with respect to Lebesgue measure and $\mathcal{B}\geq\max\{1,B_0\}$. Then, for any $M\in\mathbb{N}^+$ and the function class of ReLU multilayer perceptrons
{ $\mathcal{F}_n=\mathcal{F}_{\mathcal{D},\mathcal{W},\mathcal{U},\mathcal{S},\mathcal{B}}$}	with depth $\mathcal{D}$,  width $\mathcal{W}$ and size  $\mathcal{S} $ given by
$
 \mathcal{D}= 21(\lfloor\beta\rfloor+1)^2  M\lceil\log_2(8M)\rceil, \
 \mathcal{W} = 38(\lfloor\beta\rfloor+1)^2d^{\lfloor\beta\rfloor+1}\lceil n^{d/2(d+2\beta)}\log_2(8n^{d/2(d+2\beta)})\rceil,\
\mathcal{S}=O(n^{d/(d+2\beta)}{\color{black}(\log_2 n)^{2}}),
$
 the 
ERM $\hat{f}_n \in\arg\min_{f\in\mathcal{F}_n}L_n(f) $ satisfies

\begin{align*}
	\mathbb{E} \Vert \hat{f}_n-f_0\Vert^2_{L^2(\nu)}\leq& \Big\{c_1\mathcal{B}^5{\color{black}(\log n)^8}+324B_0^2d^{2\lfloor\beta\rfloor+\beta\vee1}M^{-4\beta/d}\Big\}(\lfloor\beta\rfloor+1)^4n^{-2\beta/(d+2\beta)},\\
	\leq&c_2\mathcal{B}^5M^{-4\beta/d}(\lfloor\beta\rfloor+1)^{4}d^{2\lfloor\beta\rfloor+\beta\vee1}n^{-2\beta/(d+2\beta)}
{\color{black}(\log n)^{8}},
\end{align*}
for $2n \ge \text{Pdim}(\mathcal{F}_n)$, where $c_1,c_2>0$ are constants which do not depend on $n,\mathcal{B},B_0,\beta$ or $M$.
\end{corollary}

	
By Corollaries \ref{c1} and \ref{c2}, the size of the \textit{deep with fixed width} network $S_{\text{DFW}}$
 and the size of the \textit{wide with fixed depth} network $S_{\text{WFD}}$ to achieve the same error rate are
\begin{equation}
\label{size1}
\mathcal{S}_{\text{DFW}}=O(n^{d/2(d+2\beta)}(\log n)) \ \text{ and } \
\mathcal{S}_{\text{WFD}}=O(n^{d/(d+2\beta)}(\log n)^{2}),
\end{equation}
respectively. So we have  the relationship $\mathcal{S}_{\text{DFW}}\approx \sqrt{\mathcal{S}_{\text{WFD}}}.$ The relative efficiency of these two networks as defined in
(\ref{re}) is
\begin{equation}
\label{re1}
\text{NRE}(\cN_{\text{DFW}}, \cN_{\text{WFD}}) = \frac{\log \mathcal{S}_{\text{WFD}}}{
\log \mathcal{S}_{\text{DFW}}}=2.
\end{equation}
Thus deep networks are twice as efficient as wide networks in terms of NRE.
In terms of sample complexity, (\ref{re1}) means that, if the sample size needed for a \textit{deep with fixed width} network to achieve the optimal error rate is $n$, then it is about $n^2$ for a \textit{wide with fixed depth} network.

Limitations of the approximation capabilities of shallow neural networks and
the advantages of deep neural networks have been well studied \citep{chui1996limitations,eldan2016power,telgarsky2016benefits}.
In \cite{telgarsky2016benefits}, it was shown that for any integer $k\geq1$ and dimension $d\geq1$, there exists a function computed by a ReLU neural network with $2k^3+8$ layers, $3k^2+12$ neurons and $4+d$ different parameters such that it cannot be approximated by networks activated by piecewise polynomial functions with no more than $k$ layers and $O(2^k)$ neurons. In addition, 
\cite{lu2017expressive} showed that
depth can be more effective than width for the expressiveness of ReLU networks.
Our calculation directly links the network structure with the sample complexity in the context of nonparametric regression.

	

{\color{black}
\begin{corollary}[Deep and wide networks]
\label{c3}
		Under model (\ref{model}), suppose that Assumptions  \ref{A1}-\ref{A2} hold, $\nu$ is absolutely continuous with respect to Lebesgue measure and $\mathcal{B}\geq\max\{1,B_0\}$. Then, for the function class of ReLU multilayer perceptrons
$\mathcal{F}_n
=\mathcal{F}_{\mathcal{D},\mathcal{W},\mathcal{U},\mathcal{S},\mathcal{B}}$		
with depth  $\mathcal{D}$, width $\mathcal{W}$ and size $\mathcal{S}$ given by
{\color{black}
\[\mathcal{W}=O(n^{d/{4(d+2\beta)}}\log_2(n)),
\mathcal{D}=O(n^{d/{4(d+2\beta)}}\log_2(n)),
\mathcal{S}=O(n^{3d/{4(d+2\beta)}}(\log n)^{4}),
\]
}
%
%
 the ERM $\hat{f}_n$ satisfies
\begin{align*}
	\mathbb{E} \Vert \hat{f}_n-f_0\Vert^2_{L^2(\nu)}\leq& \Big\{c_1\mathcal{B}^5{\color{black}(\log n)^{11}}+324B_0^2d^{2\lfloor\beta\rfloor+\beta\vee1}N^{-4\beta/d}\Big\}(\lfloor\beta\rfloor+1)^4n^{-2\beta/(d+2\beta)},\\
	\leq&c_2\mathcal{B}^5(\lfloor\beta\rfloor+1)^{4}d^{2\lfloor\beta\rfloor+\beta\vee1}n^{-2\beta/(d+2\beta)}
{\color{black}(\log n)^{11}},
\end{align*}
for $2n \ge \text{Pdim}(\mathcal{F}_n)$, where $c_1,c_2>0$ are constants which do not depend on $n,\mathcal{B},B_0$ or $\beta$.
\end{corollary}
}

By Corollary \ref{c3}, the size $\mathcal{S}_{\text{DAW}}$ of the deep and wide network achieving the optimal error bound  is
\begin{equation}
\label{size2}
\mathcal{S}_{\text{DAW}} = O(n^{3d/{4(d+2\beta)}}(\log n)^{-8}).
\end{equation}
 Combining (\ref{size1}) and (\ref{size2})  and ignoring the $\log n$ factors,
 we have
$
\mathcal{S}_{\text{DFW}}^2 \approx \mathcal{S}_{\text{WFD}}\approx \mathcal{S}_{\text{DAW}}^{4/3}.
$
Therefore, the relative efficiencies are
\[
 \text{NRE}(\cN_{\text{DFW}}, \cN_{\text{DAW}} )
=\frac{3/4}{1/2} = \frac{3}{2} \ \text{ and } \
\text{NRE}( \cN_{\text{WFD}}, \cN_{\text{DAW}})
=\frac{3/4}{1} = \frac{3}{4}.
\]
The relative sample complexity of a \textit{deep with fixed width} network versus a \textit{deep and wide} network  is $n:n^{3/2}$; and the relative sample complexity of a \textit{wide with fixed depth} network versus a \textit{deep and wide} network is $n: n^{3/4}.$

We note that the choices of the network parameters are not unique to achieve the optimal convergence rate.
For deep and wide networks, there are multiple choices that attain the optimal rate.
For example, the following two different specifications of the network parameters achieve the
same convergence rate.
\begin{align*}
	\mathcal{D}&=21(\lfloor\beta\rfloor+1)^2 \lceil n^{d/2(d+2\beta)}\log_2(8n^{d/2(d+2\beta)})\rceil, \\
	\mathcal{W}&=38(\lfloor\beta\rfloor+1)^2d^{\lfloor\beta\rfloor+1}(\log n)\lceil\log_2(8(\log n))\rceil, \
	\mathcal{S}= O(n^{d/2(d+2\beta)}(\log n)^{4}),
\end{align*}
 and
 \begin{align*}
 	\mathcal{D}&=21(\lfloor\beta\rfloor+1)^2 \lceil (\log n)\log_2(8(\log n))\rceil, \\
 	\mathcal{W}&= 38(\lfloor\beta\rfloor+1)^2d^{\lfloor\beta\rfloor+1}\lceil n^{d/2(d+2\beta)}\log_2(8n^{d/2(d+2\beta)})\rceil, \
 	\mathcal{S}= O(n^{d/(d+2\beta)}(\log n)^{4}),
 \end{align*}

The above calculations suggest that there is no unique optimal selection of network parameters for achieving the optimal rate of convergence in nonparametric regression. Instead, we should consider the efficient design of the network structure for achieving the optimal convergence rate with the minimal  network size.

\subsection{Efficient design of rectangle networks}
\label{design}
We now discuss the efficient design of \textit{rectangle networks}, i.e.,  networks with equal width for each hidden layer. For such networks with a regular shape, we have an exact relationship between the size of the network and the depth and the width:
 \begin{equation}
 \label{size-equal}
 \mathcal{S}=\mathcal{W}(d+1)+(\mathcal{W}^2+\mathcal{W})(\mathcal{D}-1)+
\mathcal{W}+1=O(\mathcal{W}^2\mathcal{D}).
\end{equation}
Based on this relationship
and Theorem \ref{thm2}, we can determine the depth and the width of the network to achieve the
optimal error with the minimal size.

Specifically, to achieve the {\color{black}optimal rate with respect to the sample size $n$} with a minimal network size, we can set
\begin{align*}
\mathcal{W}&=114(\lfloor\beta\rfloor+1)^2d^{\lfloor\beta\rfloor+1}, \
\mathcal{D}=21(\lfloor\beta\rfloor+1)^2\lceil n^{d/2(d+2\beta)}\log_2(8n^{d/2(d+2\beta)})\rceil,\\
 \mathcal{S}&=O(\mathcal{W}^2\mathcal{D})=O((\lfloor\beta\rfloor+1)^6d^{2\lfloor\beta\rfloor+2} \lceil n^{d/{2(d+2\beta)}}(\log_2 n)\rceil).
 \end{align*}
 It is interesting to note that the most efficient network's shape is a fixed-width rectangle; its width is a multiple of $d^{\lfloor\beta\rfloor+1}$, a polynomial of dimension $d$, but does not depend on the sample size $n$.
 Its depth  $\mathcal{D}=21(\lfloor\beta\rfloor+1)^2\lceil n^{d/2(d+2\beta)}\log_2(8n^{d/2(d+2\beta)})\rceil\approx O(\sqrt{n})$ for $d\gg \beta.$

 { The calculation in this subsection suggests that, in designing neural networks for high-dimensional nonparametric regression with a large $n$ and $d\gg \beta,$ we may consider setting the width of the network to be of the order
$O(d^{\lfloor\beta\rfloor+1})$
 and the depth to be proportional to $\sqrt{n}$, so as to achieve the optimal convergence rate with minimal number of network parameters.}
Qualitatively, this suggests that the depth of the network should be roughly proportional to the square root of  sample size and the width of the network should roughly be proportional to a polynomial order of the data dimension.
However,  we note that the design of a network architecture  is very much  problem specific and requires careful data-driven tuning in practice. { Also, we did not consider the optimization aspect  where deeper neural networks can be more challenging to optimize.
In general, gradient descent and stochastic gradient decent will find a reasonable solution for the optimization problem raised in deep leaning tasks with overparameterized deep networks, see for example \cite{allen2019convergence, du2019gradient} and \cite{nguyen2020rigorous}.
Also, the results here are based on the use of feedforward neural networks in the
context of nonparametric regression. In other types of problems such as image classification using
convolutional neural networks, the calculation here may not apply and new derivation is needed.

\section{Circumventing the curse of dimensionality}
\label{sec4}

	For many modern statistical and machine learning tasks, the dimension $d$ of the input data can be large,  which results in an extremely slow rate of convergence even if the sample size is big.  This problem is known as the curse of dimensionality. A promising way to mitigate the curse of dimensionality is to impose additional conditions on the data distribution and the target function $f_0$. In Lemmas \ref{lemma0} and \ref{lemma1},  the approximation error $\inf_{f\in\mathcal{F}_n}\Vert f-f_0\Vert^2_{L^2(\nu)}$ is defined with respect to the probability measure $\nu$, this provides us a chance to improve the rate.
	Although the domain of  $f_0$ is high dimensional, when the support of $X$ is concentrated on some neighborhood of a low-dimensional manifold, the upper bound of the approximation error can be much improved in terms of the exponent of the convergence rate \citep{baraniuk2009random,shen2019deep}.
	There have been growing evidence and examples indicating that  high-dimensional data tend to have low-dimensional latent structures in many applications such as image processing, video analysis, natural language processing \citep{belkin2003laplacian,hoffmann2009local}.
	
		\citet{goodfellow2016deep} argued that the approximately low-dimensional manifold assumption is
generally correct for images, supported by two observations. First, natural images are locally connected, with each image surrounded by other highly similar images reachable through image transformations (e.g., contrast, brightness). Second, natural images seem to lie on an approximately low-dimensional
structure, as the probability distribution of images is highly concentrated; uniformly sampled pixels can hardly assemble a meaningful image.  Furthermore,
results from
many numerical experiments
strongly support the low-dimensional manifold hypothesis for many image datasets \citep{roweis2000nonlinear,tenenbaum2000global,brand2002charting,fefferman2016testing}.
		For example,
		for the well-known benchmark image datasets MNIST \citep{mnist2010},
		whose ambient dimension $d=28\times 28=784$,  CIFAR-10, whose ambient dimension  $d=32\times 32 \times 3=1024$ \citep{cifar10}, and ImageNet \citep{deng2009imagenet}, whose ambient dimension $d=224\times 224\times 3=150,528$, the estimated intrinsic dimensions of these three datasets are between 9 and 43 \citep{pope2020intrinsic,recanatesi2019dimensionality}.
Therefore, it is important
to study the properties deep nonparametric regression under the assumption that the intrinsic dimension is lower than its ambient dimension.

In this section, we establish non-asymptotic error bounds for the ERM $\hat{f}_n$ under three different cases of low-dimensional support of $X$: (a) an approximate low-dimensional manifold; (b) an  exact low-dimension manifold;  and (c) a low Minkowski dimension set. Case (a) is a realistic assumption. Case (b) is of theoretical interest, since in this case we can show that the convergence rate is determined by the exact dimension of the manifold. Case (a) is more difficult than (b) in the sense that the convergence rate under (a) is slower than that under (b). The Minkowski dimension is a more general notion than the topological dimension of a manifold. In particular, case (c) includes (b) as a special case, but does not include (a).  Since the Minkowski dimension only depends on the metric, it can also be used to measure the dimensionality of highly non-regular sets \citep{falconer2004fractal}.

\subsection{Approximate low-dimensional manifold assumption}

The assumption that high-dimensional data tend to lie in the vicinity of a low-dimensional manifold is the basis of manifold learning
\citep{fefferman2016testing}. It is also one of the basic assumptions in semi-supervised learning \citep{belkin2004semi}.
In applications, one rarely observes data that are located on an exact manifold. It is more reasonable to assume that they are concentrated on a neighborhood of a low-dimensional manifold.
For instance, the empirical studies
by
\cite{carlsson2009topology}
suggest that image data tend to have low intrinsic dimensions
and  be supported on approximate lower-dimensional manifolds.
We formally state the approximate low-dimensional manifold support assumption below.

\begin{assumption}
	\label{A3}
	The predictor $X$ is supported on $\mathcal{M}_\rho$, a $\rho$-neighborhood of $\mathcal{M}\subset[0,1]^d$, where $\mathcal{M}$ is a compact $d_\mathcal{M}$-dimensional Riemannian submanifold \citep{lee2006riemannian} and
	$$\mathcal{M}_\rho=\{x\in[0,1]^d: \inf\{\Vert x-y\Vert_2: y\in\mathcal{M}\}\leq \rho\}, \ \rho \in (0, 1).$$
\end{assumption}


The following theorem gives excess risk bounds under Assumption \ref{A3} and other appropriate conditions.

\begin{theorem}[Non-asymptotic error bound] \label{thm3}
	Under model (\ref{model}), suppose that Assumptions \ref{A1}-\ref{A3} hold, the probability measure $\nu$ of $X$ is absolutely continuous with respect to the Lebesgue measure and $\mathcal{B}\geq\max\{1,B_0\}$. Then for any $N, M\in\mathbb{N}^+$, the function class of ReLU multi-layer perceptrons $\mathcal{F}_n=\mathcal{F}_{\mathcal{D},\mathcal{W},\mathcal{U},\mathcal{S},\mathcal{B}}$ with width $\mathcal{W}=38(\lfloor\beta\rfloor+1)^2d_\delta^{\lfloor\beta\rfloor+1}N\lceil\log_2(8N)\rceil$ and depth $\mathcal{D}=21(\lfloor\beta\rfloor+1)^2M\lceil\log_2(8M)\rceil$ , the prediction error of the empirical risk minimizer $\hat{f}_n$	satisfies
	$$\mathbb{E} \Vert \hat{f}_n-f_0\Vert^2_{L^2(\nu)}\leq C_1\mathcal{B}^5\frac{\mathcal{S}\mathcal{D}\log(\mathcal{S})
		(\log n)^5}{n}+\frac{(36+C_2)^2B_0^2}
	{(1-\delta)^{2\beta}}(\lfloor\beta\rfloor+1)^4dd_\delta^{3\lfloor\beta\rfloor}(NM)^{-4\beta/d_\delta}$$
	for $n \ge \text{Pdim}(\mathcal{F}_n)/2$ and $\rho\leq C_2(NM)^{-2\beta/d_{\delta}}(s+1)^2d^{1/2}d_\delta^{3s/2}(\sqrt{{d}/{d_\delta}}+1-\delta)^{-1}(1-\delta)^{1-\beta}$, where $d_\delta=O(d_\mathcal{M}{\log(d/\delta)}/{\delta^2})$ is an integer  such that $d_\mathcal{M}$$\leq d_\delta<d$ for any $\delta\in(0,1)$, and $C_1,C_2>0$ are constants that do not depend on $n,\mathcal{B},\mathcal{S},\mathcal{D},B_0,\beta,\rho,\delta,N$ or $M$.

\end{theorem}

As in Subsection \ref{efficiency}, to achieve the optimal convergence rate with a minimal network size,  we can set  $\mathcal{F}_n=\mathcal{F}_{\mathcal{D},\mathcal{W},\mathcal{U},\mathcal{S},\mathcal{B}}$ to consist of fixed-width networks with
$ \mathcal{W}=114(\lfloor\beta\rfloor+1)^2d_\delta^{\lfloor\beta\rfloor+1}, \
	\mathcal{D}=21(\lfloor\beta\rfloor+1)^2\lceil n^{d_\delta/2(d_\delta+2\beta)}\log_2(8n^{d_\delta/2(d_\delta+2\beta)})\rceil,\ \mathcal{S}=O(\mathcal{W}^2\mathcal{D})=O((\lfloor\beta\rfloor+1)^6d_\delta^{2\lfloor\beta\rfloor+2} \lceil n^{d_\delta/{2(d_\delta+2\beta)}}(\log_2 n)\rceil).
$
Then the prediction error of $\hat{f}_n$ in Theorem \ref{thm3} becomes
\begin{align}
	\label{cursea}
	\mathbb{E} \Vert \hat{f}_n-f_0\Vert^2_{L^2(\nu)}&\leq C_3(1-\delta)^{-2\beta}\mathcal{B}^5dd_\delta^{3\lfloor\beta\rfloor+3}
	(\lfloor\beta\rfloor+1)^9n^{-2\beta/(d_\delta+2\beta)}
	{\color{black}(\log n)^8},
\end{align}
where $C_3>0$ is a constant not depending on $n,d,d_\delta,\mathcal{B},\mathcal{S},\mathcal{D},B_0,\delta$ or $\beta$.

We can also consider the relative efficiencies of networks with different shapes in a way completely
similar to those in Section \ref{efficiency}.

Theorem \ref{thm3} shows that nonparametric regression using deep neural networks can alleviate the curse of dimensionality under an approximate
manifold assumption.
This is different from the hierarchical structure assumption on $f_0$  \citep{bauer2019deep, schmidt2020nonparametric}.
We note that under the \textit{approximate} manifold assumption, the dimension of the support of $X$ is still $d$ and only shrinks to $d_{\mathcal{M}}$. The convergence rate  in (\ref{cursea})
depends on $d_\delta=O(d_\mathcal{M}\log(d))$, which is smaller than
$d$ but still greater than $d_{\mathcal{M}}$ with an extra $\log (d)$ factor. Intuitively, this $\log (d)$ factor is due to the fact that the dimension of the \textit{approximate} manifold is still $d.$
It is not clear if it is possible to remove the effect of $d$ on the convergence rate under the approximate low-dimensional manifold assumption. This is a technically challenging problem and deserves further study in the future.


\subsection{Exact low-dimensional manifold assumption}
{\color{black}
	Under the exact manifold support assumption, we show that the $\log (d)$ factor in (\ref{cursea}) can be removed. We establish error bounds that
	achieve the minimax optimal convergence rate with a prefactor only depending linearly on the ambient dimension $d$.
	
	\begin{assumption}
		\label{A4}
		The predictor $X$ is supported on $\mathcal{M}\subset[0,1]^d$, where a $\mathcal{M}$ is a compact $d_{\mathcal M}$-dimensional Riemannian manifold isometrically embedded in $\mathbb{R}^d$ with condition number $(1/\tau)$ and area of surface $S_\mathcal{M}.$
	\end{assumption}

For a compact Riemannian manifold $\mathcal{M}$, the condition number $(1/\tau)$ controls both local properties of the manifold (such as curvature) and global properties (such as self-avoidance) \citep{baraniuk2009random}. Some authors refers to $\tau$ as the geometric concept ``reach" \citep{federer1959curvature,aamari2019estimating}, which is the largest number having the following property: The open normal bundle about $\mathcal{M}$ of radius $r$ is embedded in $\mathbb{R}^d$ for all $r<\tau$ \citep{niyogi2008finding,baraniuk2009random}. Intuitively, at each point $x\in\mathcal{M}$, the radius of the osculating circle is no less than $\tau$, where a large $\tau$ prevents the manifold $\mathcal{M}$ to be curvy.
Condition number $(1/\tau)$  or the reach $\tau$ here influences the complexity of function approximation on $\mathcal{M}$ using neural networks.
	
	The surface area
$S_\mathcal{M}$ of a manifold $\mathcal{M}$ is defined as the integral of 1 over the manifold with respect to the Riemannian volume element (Chapter 10, \citet{book:206368}; Chapter 8, \citet{lee2006riemannian}; and Chapter 5, \citet{hubbard2015vector}). For example,
for the
surface area  of
a $d$-dimensional unit ball, this definition gives the well-know result $2\pi^{d/2}/\Gamma(d/2),$ where $\Gamma$ is the gamma function.
	For function approximation on $\mathcal{M}$ by neural networks, we approximate the function on a finite number of charts which cover  $\mathcal{M}$. Larger surface area $S_\mathcal{M}$ only leads to a larger number of charts, which further leads to a wider (linearly in $S_\mathcal{M}$) neural network width and larger prefactor of the approximation error.


	
	\begin{theorem}[Non-asymptotic error bound]\label{thm5}
		Under model (\ref{model}), suppose that Assumptions \ref{A1}-\ref{A2} and \ref{A4} hold, and $\mathcal{B}\geq\max\{1,B_0\}$. Then for any $N, M\in\mathbb{N}^+$, the function class of ReLU multi-layer perceptrons $\mathcal{F}_n=\mathcal{F}_{\mathcal{D},\mathcal{W},\mathcal{U},\mathcal{S},\mathcal{B}}$ with $\mathcal{W}=266(\lfloor\beta\rfloor+1)^2\lceil S_\mathcal{M}(6/\tau)^{d_{\mathcal M}}\rceil(d_{\mathcal M})^{\lfloor\beta\rfloor+2}N\lceil\log_2(8N)\rceil$ and depth $\mathcal{D}=21(\lfloor\beta\rfloor+1)^2M\lceil \log_2(8M)\rceil+2d_{\mathcal M}+2$, the prediction error of the empirical risk minimizer $\hat{f}_n$ satisfies
		\begin{align*}
			\mathbb{E} \Vert \hat{f}_n-f_0\Vert^2_{L^2(\nu)}\leq C_1\mathcal{B}^5\frac{\mathcal{S}\mathcal{D}\log(\mathcal{S})
				(\log n)^5}{n}+C_2{B_0^2}(\lfloor\beta\rfloor+1)^4d(d_{\mathcal M})^{3\lfloor\beta\rfloor+1}(NM)^{-4\beta/d_{\mathcal M}}
		\end{align*}
		for $n \ge \text{Pdim}(\mathcal{F}_n)/2,$ where $C_2>0$ is a constant independent of
		$n,d,d_{\mathcal M},\mathcal{B},\mathcal{S},\mathcal{D},N,M,\beta,B_0,\tau$ and $S_\mathcal{M}$.
		Furthermore, if we set $\mathcal{F}_n=\mathcal{F}_{\mathcal{D},\mathcal{W},\mathcal{U},\mathcal{S},\mathcal{B}}$ to consist of fixed-width networks with
		\begin{eqnarray*}
			\mathcal{W}&=&798(\lfloor\beta\rfloor+1)^2\lceil S_\mathcal{M}(6/\tau)^{d_{\mathcal M}}\rceil(d_{\mathcal M})^{\lfloor\beta\rfloor+2},\\
			\mathcal{D}&=&21(\lfloor\beta\rfloor+1)^2\lceil n^{d_{\mathcal M}/2(d_{\mathcal M}+2\beta)}\log_2(8n^{d_{\mathcal M}/2(d_{\mathcal M}+2\beta)})\rceil+2d_{\mathcal M}+2,\\
			\mathcal{S}&=&O( (\lfloor\beta\rfloor+1)^6d(6/\tau)^{2d_{\mathcal M}}(d_{\mathcal M})^{2\lfloor\beta\rfloor+5} n^{d_{\mathcal M}/{2(d_{\mathcal M}+2\beta)}}\log_2(n)),
		\end{eqnarray*}
		the prediction error of $\hat{f}_n$ satisfies
		\begin{align*}
			\mathbb{E} \Vert \hat{f}_n-f_0\Vert^2_{L^2(\nu)}&\leq C_3\mathcal{B}^5(\lfloor\beta\rfloor+1)^9(6/\tau)^{2d_{\mathcal M}}(d_{\mathcal M})^{3\lfloor\beta\rfloor+6}d(\log n)^8n^{-2\beta/(d_{\mathcal M}+2\beta)},
		\end{align*}
		where $C_3>0$ is a constant independent of
		$n,d,d_{\mathcal M},\mathcal{B},B_0,\beta,\tau$ and $S_\mathcal{M}$.
	\end{theorem}
	
	Theorem \ref{thm5} shows that the ERM $\hat{f}_n$
	achieves the optimal minimax rate $n^{-2\beta/(d_{\mathcal M}+2\beta)}$   up to a logarithmic factor
	under the exact manifold assumption.
	Under
	this assumption, the optimal
	rate
	up to a logarithmic factor
	has also been obtained by \citet{chen2019nonparametric} and  \citet{schmidt2019deep}.
	Our result differs from these previous ones in two important aspects. {\color{black} First,
		the prefactor in the error bound depends on the
		ambient dimension $d$ linearly instead of exponentially. } Second,
	the network structure in our result can be more flexible,  which does not need
	to be fixed-width or fixed-depth.
	Moreover, in our proof of Theorem \ref{thm5}, we apply linear coordinate maps instead of smooth coordinate maps used in the existing work.
	An attractive property of  linear coordinate maps is that they can be exactly
	represented
	by ReLU shallow networks without error. We also weaken the regularity conditions,
	we do not require the smoothness index of each coordinate map and the functions in the partition of unity to be
	$\beta d/d_{\mathcal M}$,  which depends on the ambient dimension $d$ and can be large.
}

\subsection{Low Minkowski dimension assumption}
{\color{black}
	Lastly,
	we consider the important case
	when data is supported on a set with low Minkowski dimension \citep{bishop_peres2016} and
	obtain fast convergence rates.
	
	\begin{definition}[Minkowski dimension]
		The upper and  lower Minkowski dimension of a set $A\subseteq \mathbb{R}^d$ are defined respectively as
		\begin{align*}
			\overline{\rm dim}_M(A):=\limsup_{\varepsilon\to0}\frac{\log\mathcal{N}(\varepsilon,\Vert\cdot\Vert_2,A)}
			{\log(1/\varepsilon)},\
			\underline{\rm dim}_M(A):=\liminf_{\varepsilon\to0}\frac{\log\mathcal{N}(\varepsilon,\Vert\cdot\Vert_2,A)}
			{\log(1/\varepsilon)}.
		\end{align*}
		If $\overline{\rm dim}_M(A)=\underline{\rm dim}_M(A)={\rm dim}_M(A)$, then ${\rm dim}_M(A)$ is called the Minkowski dimension of the set $A$.
	\end{definition}
	For simplicity, we denote $d^*={\rm dim}_M(A)$ below.
	The Minkowski dimension measures how the covering number of a set $A$ grows when the radius of the covering balls converges to zero. When $A$ is a manifold, its Minkowski dimension is the same as the dimension of the manifold. Since the Minkowski dimension only depends on the metric, it can
	be used to measure the dimensionality of highly non-regular sets such as fractals \citep{falconer2004fractal}. \cite{nakada2020adaptive} showed that deep neural networks can adapt to the low-dimensional structure of data, and the convergence rates do not depend on the nominal high dimensionality of data, but on its lower intrinsic Minkowski dimension.
	Based on random projection, the curse of dimensionality can also be lessened when data is supported on a set with low Minkowski dimension.

	\begin{theorem}[Non-asymptotic error bound] \label{thm4}
		Under model (\ref{model}), suppose that Assumptions \ref{A1}-\ref{A2} hold, $\mathcal{B}\geq\max\{1,B_0\}$ and $X$ is supported on a set $A\subseteq[0,1]^d$ with Minkowski dimension $d^*\equiv {\rm dim}_M(A)<d$ . Then for any $N, M\in\mathbb{N}^+$, the function class of ReLU multi-layer perceptrons $\mathcal{F}_n=\mathcal{F}_{\mathcal{D},\mathcal{W},\mathcal{U},\mathcal{S},\mathcal{B}}$ with width $\mathcal{W}=38(\lfloor\beta\rfloor+1)^23^{d_0}d_0^{\lfloor\beta\rfloor+1}N\lceil\log_2(8N)\rceil$ and depth $\mathcal{D}=21(\lfloor\beta\rfloor+1)^2M\lceil \log_2(8M)\rceil+2d_0$, the prediction error of the empirical risk minimizer $\hat{f}_n$	satisfies,
		\begin{align*}
			\mathbb{E} \Vert \hat{f}_n-f_0\Vert^2_{L^2(\nu)}\leq & C_1\mathcal{B}^5\frac{\mathcal{S}\mathcal{D}\log(\mathcal{S})
				(\log n)^5}{n}\\
			& +C_2\frac{B_0^2}{(1-\delta)^{\beta}}(\lfloor\beta\rfloor+1)^4
			d_0^{2\lfloor\beta\rfloor+\beta\vee1+1}d (NM)^{-4\beta/d_0}
		\end{align*}
		for $n \ge \text{Pdim}(\mathcal{F}_n)/2$, 	where $d\ge d_0\ge \kappa d^*/\delta^2=O(d^*/\delta^2)$ for $\delta\in(0,1)$ and some constant $\kappa>0$, and $C_1,C_2>0$ are constants not depending on $n,\mathcal{B},\mathcal{S},\mathcal{D},B_0,\beta,\kappa,\delta,N$ or $M$.
	\end{theorem}
	
	As discussed in Subsection \ref{efficiency}, to achieve the optimal convergence rate with a minimal network size,  we can set $\mathcal{F}_n=\mathcal{F}_{\mathcal{D},\mathcal{W},\mathcal{U},\mathcal{S},\mathcal{B}}$ to consist of fixed-width networks with
	\begin{align*}
		\mathcal{W}&=114(\lfloor\beta\rfloor+1)^23^{d_0}d_0^{\lfloor\beta\rfloor+1},\
		\mathcal{D}=21(\lfloor\beta\rfloor+1)^2\lceil n^{d_0/2(d_0+2\beta)}\log_2(8n^{d_0/2(d_0+2\beta)})\rceil,\\
		\mathcal{S}&=O(\mathcal{W}^2\mathcal{D})=O( (\lfloor\beta\rfloor+1)^63^{2d_0}d_0^{2\lfloor\beta\rfloor+2}\lceil n^{d_0/{2(d_0+2\beta)}}(\log n)\rceil).
	\end{align*}
	Then, the prediction error of $\hat{f}_n$ in Theorem \ref{thm4} is
	\begin{align*}
		\mathbb{E} \Vert \hat{f}_n-f_0\Vert^2_{L^2(\nu)}&\leq C_3(1-\delta)^{-\beta}\mathcal{B}^53^{3d_0}
		d_0^{3\lfloor\beta\rfloor+3}(\lfloor\beta\rfloor+1)^9 d n^{-2\beta/(d_0+2\beta)}(\log n)^8,
	\end{align*}
	where $C_3>0$ is a constant not depending on $n,d,d_0,\mathcal{B},\mathcal{S},\mathcal{D},B_0,\delta$ or $\beta$.


Prior to this work,  \citet{nakada2020adaptive}
obtained an error bound with convergence rate $n^{-2\beta/(d^{\#}+2\beta)}$ up to $\log n$ factor for a $d^{\#}>{\rm dim}_M(A)=d^*$ where $d^{\#}$ can be arbitrarily close to the Minkowski dimension $d^*$ of the support of the data. While our obtained convergence rate is $n^{-2\beta/(d_0+2\beta)}$ up to a $\log n$ factor  for $d_0=O(d^*/\delta^2)$ with $\delta\in(0,1)$.  {\color{black} The convergence rate of \cite{nakada2020adaptive} can be faster than that of ours.} The prefactor in the error bound of \cite{nakada2020adaptive}
is  $O(d^{d^*}+5^d)$, while ours is
 $O(d9^{d^*}d^*{^{3\lfloor\beta\rfloor+3}}),$ which can be much smaller.
	In  their proof of the approximation result (Theorem 5 of \cite{nakada2020adaptive}),
	the minimum set of hypercubes covering the support of $X$ is partitioned into $5^d$ subsets. Within each subset, the hypercubes are separated by a constant distance from each other. For each such subset, a trapezoid-type deep neural network approximates the Taylor expansion of $f_0$ locally. Then a large neural network combining these local approximators is used to realize the whole approximation on the support of $X$.
    To ensure an overall $\varepsilon$ approximation error,
	the network size must be $C_1\varepsilon^{-d^{\#}/\beta}+C_2$, where $C_1=2[(50d+17)d^{d^{\#}}(3M)^{d^{\#}/\beta}c_1+2d\{11+(1+\beta)/d^{\#}\}c_2\{2^{d^{\#}/\beta}+c_3d^{d^{\#}}(3M)^{d^{\#}/\beta}\}]=O(d^{d^{\#}}3^{d^{\#}/\beta})$ for some constants $c_1,c_2,c_3,M>0$ and $C_2=2[12+42*5^d+2d+2d\{11+(1+\beta)/d^{\#}\}(1+\lceil\log_2\beta\rceil)]=O(5^d)$; and, these prefactors of the network size, which could be
	large for moderate $d$ or $d^{\#}$, will lead to a large prefactor of the overall non-asymptotic error bound.
	
    In comparison, in Theorem \ref{thm4} we allow relatively more flexible network shapes and the network width could be a multiple of $3^{d_0}d_0^{\lfloor\beta\rfloor+1}$ rather than $d^{d_0}$ or $5^d$, to achieve a $9^{d_0}d d_0^{3\lfloor\beta\rfloor+3}$ prefactor of the generalization error bound.
	
	In our proof of Theorem \ref{thm4}, we leverage a generalized Johnson-Lindenstrauss lemma for infinite sets (see, for example, Theorem 13.15 in \cite{boucheron2013concentration})
	to project the closure of the support of $X$ into lower-dimensional space. Then our newly proved approximation result  Theorem \ref{thm_apx} is applied in the lower-dimensional space, which is in terms of a smaller effective dimensionality related to the Minkowski dimension of the support of $X$.
	The projection is approximately a linear transformation and can be exactly represented by a three-layer ReLU network,  thus it causes no approximation error.
	In addition, this also avoids the $5^d$ prefactor in the formula of error bounds or the network width.
	
	Finally, we note that the results of  \cite{nakada2020adaptive} and Theorem \ref{thm4}  do not cover  Theorem \ref{thm3}, nor vice versa. On one hand, an approximate manifold assumption allows a closed ball or a sphere in $\mathbb{R}^d$ contained in the support of $X$, in which case the Minkowski dimension of such approximate low-dimensional manifold is $d$ and no faster convergence rate can be obtained. To see this, if a closed ball $\mathbb{B}(a)$ (or a sphere) with radius $a>0$ in $\mathbb{R}^d$ is contained in $A\subseteq[0,1]^d$, the support of $X$, then
	the $\epsilon$-covering number of $A$ is no less than $(a/\epsilon)^d$
	(see e.g., Corollary 4.2.13 in \cite{vershynin2018high}),
	which implies that the Minkowski dimension of $A$ is $d$. On the other hand, the Minkowski dimension can be used to measure non-smooth low-dimensional set such as fractals which may not be a low-dimensional manifold or a neighborhood of a low-dimensional manifold.
	
}

\section{Related works}	
\label{review}	In this section, we discuss the connections and differences between our
 work and the related works with respect to the non-asymptotic error bounds, 	the structural assumptions on the target regression function $f_0$,  and the distributional assumptions on the data. 	
	
 \subsection{Error bounds}
 {
		\label{bound}
Recently, \citet{bauer2019deep}, \citet{schmidt2020nonparametric} and \citet{farrell2021deep}
studied the convergence properties of nonparametric regression using feedforward neural networks.
\citet{bauer2019deep} required that the activation function satisfies certain smoothness conditions; \citet{schmidt2020nonparametric} and \citet{farrell2021deep} considered the ReLU activation function. \citet{bauer2019deep} and \citet{schmidt2020nonparametric} assumed that the regression function has a composition structure similar.
They showed that nonparametric regression using feedforward neural networks with
a polynomial-growing network width $\mathcal{W}=O(d^\beta)$ achieves the optimal rate of convergence \citep{stone1982optimal} up to a $\log n$ factor,  however, with a prefactor
$C_d=O(a^d)$ for some $a \ge 2,$
unless the network width $\mathcal{W}=O(a^d)$ and size $\mathcal{S}=O(a^d)$ grow exponentially as $d$ grows.

{
A key difference between our work and the existing results
is in  how the prefactor $C_d$ depends on $d$.
Specifically, 
the prefactor $C_d$ in our results depends polynomially on $d$
and involves $d^\beta$ as a linear factor.
 In comparison, the prefactor $C_d$ in the error bounds obtained by \citet{bauer2019deep}, \citet{schmidt2020nonparametric}, \citet{farrell2021deep} and others depends on $d$ exponentially. For high-dimensional data with a large $d$, it is not clear when such an error bound is useful in a non-asymptotic sense.
Similar concerns about this type of error bounds as established in \citet{schmidt2020nonparametric}
are raised in the discussion by \citet{ghorbani2020discussion},
who looked at the example of additive models and pointed out that in the
upper bound of the form { $\mathbb{E}\Vert\hat{f}_n-f_0\Vert^2_{L^2(\nu)} \le C(d) n^{-\epsilon_*} \log^2 n$ for some $\epsilon_*>0$} obtained in  \citet{schmidt2020nonparametric},  the $d$-dependence of the prefactor $C(d)$ is not characterized. It also assumes $n$ large enough, that is, $n \geq n_0(d)$ for an unspecified $n_0(d)$. They further pointed out that using the proof technique in the paper,
it requires $n  \gtrsim d^{d}$ for the error bound to hold in the additive models.
For large $d$, such a sample size requirement is
difficult to be satisfied in practice.
}
Another important difference between our results and the existing ones is that
our error bounds are given explicitly in terms of the width and the depth of the network.
This is more informative than the results characterized by just the network size.
Such an explicit error bound can provide guidance to the design of networks. For example, we are able to provide more insights into how the error bounds depend on the network structures, as given in  Corollaries \ref{c1}-\ref{c3} in Section \ref{efficiency}.

	
Finally,  in contrast to the results of \citet{gyorfi2006distribution} and \cite{farrell2021deep}, we do not make the boundedness assumption on the response $Y$ and only assume $Y$ to be sub-exponential. 	\citet{bauer2019deep} assumes that $Y$ is sub-Gaussian.
\citet{schmidt2020nonparametric} assumes i.i.d. normal error terms and requires the network parameters (weights and bias) to be bounded by $1$ and satisfy a sparsity constraint, which is not
the usual practice in the training of neural network models in applications.

\subsection{
Structural assumptions on the regression function}
\label{review:true}
 A well-known semiparametric model for  mitigating the curse of dimensionality is the single index model
	$f_0(x)=g(\theta^\top x),  \quad {x\in\mathbb{R}^d},$
where $g:\mathbb{R}\to\mathbb{R}$ is a univariate function and $\theta\in\mathbb{R}^d$ is a $d$-dimensional vector
\citep{hardle1993optimal,
horowitz1996, 
kong2007variable}.
A generalization of the single index model is
	$f_0(x)=\sum_{k=1}^K g_k(\theta_k^\top x),  \quad x\in\mathbb{R}^d,$
	where $K\in\mathbb{N}$, $g_k:\mathbb{R}\to\mathbb{R}$ and $\theta_k\in\mathbb{R}^d$ \citep{friedman1981projection}. In these models, the rate of convergence can be $n^{-2\beta/(2\beta+1)}$ up to some logarithmic factor if the univariate functions $g_k (\cdot)$ are $\beta$-H\"older smooth.
Another well-known model is the additive model
\citep{stone1986dimensionality}
	$f_0(x_1, \ldots,x_d)=f_{0,1}(x_1)+\cdots +f_{0,d}(x_d),  \quad {x=(x_1, \ldots, x_d)^\top \in\mathbb{R}^d}.$
For $\beta$-H\"older smooth univariate functions $f_{0,1}, \ldots,f_{0,d}$,
\cite{stone1982optimal} showed that the optimal minimax rate of convergence is $n^{-2\beta/(2\beta+1)}$.  \cite{stone1994use} also generalized the additive model to an interaction model
	$f_0(x)=\sum_{I\subseteq\{1,...,d\},\vert I\vert=d^*}f_I(x_I),  \quad x=(x_1,\ldots,x_d)^\top \in\mathbb{R}^d, $
	where $d^*\in\{1,\ldots,d\}$, $I=\{i_1,\ldots,i_{d^*}\}$,  $1\leq i_1<\ldots<i_{d^*}\leq d$, $x_I=(x_{i_1}, \ldots, x_{i_{d^*}})$ and all $f_I$ are $\beta$-H\"older smooth functions defined on $\mathbb{R}^{\vert I\vert}$. In this model, the optimal minimax rate of convergence was proved to be $n^{-2\beta/(2\beta+d^*)}$.
	
	 \citet{yang2015minimax} studied the minimax-optimal nonparametric regression under the so-called sparsity inducing condition, under which $f_0$ depends on a small subset of $d^*$ predictors with $d^*\leq\min\{n,d\}$. Under this assumption, for a $\beta$-H\"older smooth function $f_0$ and  continuously distributed $X$ with a bounded density on $[0,1]^d$, they proved that the prediction error is of the order $O(c_1 n^{-2\beta/(d^*+2\beta)}+c_2 \log(d/d^*)d^*/n).$  \citet{yang2015minimax} noted that, under the sparsity inducing assumption, the estimation still suffers from the curse of dimensionality in the large $d$ small $n$ settings, unless $d^*$ is substantially smaller than $d$.

For sigmoid or bounded continuous activated deep regression networks, \citet{bauer2019deep} showed that  the curse of dimension can be circumvented by assuming that $f_0$ satisfies the $\beta$-H\"older smooth {\it generalized hierarchical interaction model} of order $d^*$ and level $l$.
	Under such a structural assumption,  the target function $f_0$ is essentially a composition of multi-index model and $d^*$-dimensional smooth functions.
 \citet{bauer2019deep} showed that the convergence rate of the prediction error with this assumption achieves
 $(\log n)^3n^{-2\beta/(2\beta+d^*)}$.
For the ReLU activated deep regression networks, \citet{schmidt2020nonparametric} alleviated the curse of dimensionality by assuming that
$f_0$ is a composition of a sequence of  functions:
	$f_0=g_q\circ g_{q-1}\circ \cdots \circ g_1\circ g_0$
	with $g_i:[a_i,b_i]^{d_i}\to[a_{i+1},b_{i+1}]^{d_{i+1}}$ and $\vert a_i\vert$, $\vert b_i\vert\leq K$ for some positive $K$ and all $i$. For each $g_i=(g_{ij})^\top_{j=1,\ldots,d_{i+1}}$ with $d_{i+1}$ components, let $t_i$ denote the maximal number of variables on which each of the $g_{ij}$ depends on, and it is assumed that each $g_{ij}$ is a $t_i$-variate function belonging to
the ball of $\beta_i$-H\"older smooth functions with radius $K$,
The convergence rate is $\phi_n=\max_{i=0,\ldots,q} n^{-{2\beta_i^*}/{(2\beta_i^*+t_i)}},$
where	$\beta_i^*=\beta_i\Pi_{\ell=i+1}^q\min\{\beta_\ell,1\}.$
The resulting rate of convergence is shown to be $C_d (\log n)^3\phi_n.$
However,
 the prefactor $C_d$ in these results may depend on $d$ exponentially.
	
Recently, \cite{kohler2019estimation} assumed that the regression function $f_0$ has a  locally low dimensionality $d^*$ and obtained results that can circumvent the curse of dimensionality.
	Since such a function $f$
is generally not globally smooth, not even continuous, \cite{kohler2019estimation} assumed the true target function $f_0$ is bounded between two functions with low local dimensionality.
Under the $\beta$-H\"older smoothness assumption on $f_0$, proper distributional assumptions on $X$ and other suitable conditions, they showed that the prediction error of networks with the \textit{sigmoidal activation function} can attain the rate
$(\log n)^3 n^{-2\beta/(d^*+2\beta)}.$
	
\subsection{Assumptions on the support of data distribution}
	\label{review:data}
There have been growing evidence and examples indicating that  high-dimensional data tend to have low-dimensional latent structures in many
applications such as image processing, video analysis, natural language processing \citep{belkin2003laplacian,hoffmann2009local,nakada2020adaptive}.
There has been a great deal of efforts  to
deal with the curse of dimensionality by assuming that the data of concern lie on an embedded manifold within a high-dimensional space, e.g.,  kernel methods (\cite{kpotufe2013adaptivity}), $k$-nearest neighbor(\cite{kpotufe2011k}), local regression (\cite{bickel2007local,cheng2013local,aswani2011regression}), Gaussian process regression (\cite{yang2016bayesian}), and deep neural networks (\cite{nakada2020adaptive,schmidt2019deep,chen2019efficient,chen2019nonparametric}). Many studies have focused on representing the data on the manifold itself, e.g.,  manifold learning or dimensionality reduction (\cite{pelletier2005kernel,hendriks1990nonparametric,tenenbaum2000global,
donoho2003hessian,belkin2003laplacian,lee2007nonlinear}).
 Once the data can be mapped into a lower-dimensional space or well represented,  the curse of dimensionality can be mitigated.
Recently, several authors considered nonparametric regression using neural networks with  a low-dimensional manifold support assumption \citep{chen2019efficient,chen2019nonparametric,schmidt2019deep,
cloninger2020relu, nakada2020adaptive}.
In \cite{chen2019nonparametric}, they focus on the estimation of the target function $f_0$ on a bounded $d^*$-dimensional compact Riemannian manifold isometrically embedded in $\mathbb{R}^d$.  When $f_0$ is assumed to be $\beta$-H\"older smooth, approximation rate with ReLU networks for $f_0$ was derived. The resulting prediction error is of the rate $O(n^{-2\beta/(d^*+2\beta)}(\log n)^3),$ when the network class $\mathcal{F}_{\mathcal{D},\mathcal{U},\mathcal{W},\mathcal{S},\mathcal{B}}$ is properly designed with depth $\mathcal{D}=O(\log n)$, width $\mathcal{W}=O(n^{d^*/(2\beta+d^*)})$, 
size $\mathcal{S}=O(n^{d^*/(2\beta+d^*)}\log n)$ and each parameter is bounded by a given constant.
{Under similar assumptions, \cite{nakada2020adaptive}
established the approximation rate with deep ReLU networks for $f_0$ defined on a  set 
with a low Minkowski dimension. Their rate is in terms of Minkowski Dimension $d^*_0.$
The Minkowski dimension can describe a
broad class of low dimensional sets where the manifold needs not to be smooth. The relation between the Minkowski dimension and other dimensions can be found in
\cite{nakada2020adaptive}.
{\color{black} Similar convergence rates were obtained by \cite{schmidt2019deep} in terms of the manifold dimension under the exact manifold support assumption.}  Our Theorem \ref{thm5} reduces the  exponentially dependence  of the prefactor on $d$ in these previous works  into
linearly allowing   more flexible network structures.

Theorem \ref{thm3} differs from the aforementioned existing results in several aspects.
First, these existing results assume that the distribution of $X$ is supported on an exact low-dimensional manifold or a set with low Minkowski dimension, whereas in Theorem \ref{thm3} we assume that it is supported on an approximate low-dimensional manifold, whose Minkowski dimension can be the same as that of the ambient space $d$.
Second, the size $\mathcal{S}$ of the network or the nonzero weights and bias need to grow at the rate of $2^{d_{\mathcal{M}}}$ with respect to the dimension ${d_{\mathcal{M}}}$ in many existing results.  The term  $2^{{d_{\mathcal{M}}}}$ will dominate the prefactor
in  the excess risk bound,
which could destroy the bound even when the sample size $n$ is large. In comparison, our error bound depends  on ${d_{\mathcal{M}}}$ polynomially through $({d_{\mathcal{M}}\log d})^{3\lfloor\beta\rfloor+3}$ in the approximate manifold case.
Third, to achieve the optimal rate of convergence, the network shape is generally limited to certain types such as a fixed-depth network in \cite{nakada2020adaptive} or a network with depth $\mathcal{D}=O(\log n)$ in \citet{schmidt2019deep} and \citet{chen2019nonparametric},
while we allow relatively more flexible network designs. Moreover,  our assumptions on the data distribution are
weaker as discussed earlier.
Lastly, in Theorem \ref{thm4} we derived an error bound with a
convergence rate $n^{-2\beta/(2\beta+d_0)}$ with $d_0=O(d^*)$ in terms of the Minkowski dimension $d^*$, which alleviates the curse of dimensionality. As discussed  below Theorem \ref{thm4}, we used a different argument based on a generalized Johnson-Lindenstrauss lemma
for dimension reduction in our proof from that of \cite{nakada2020adaptive}.
We allow a relatively more flexible network architecture and achieve an improved prefactor in the excess risk bound.

}

\section{Conclusions} 
\label{sec5}

In this paper, we have established  neural network approximation error bounds with polynomial prefactors
for  H\"older smooth functions and  non-asymptotic excess risk bounds for deep nonparametric regression.
We have also derived new non-asymptotic excess risk bounds under \text{manifold assumptions}, including an approximate low-dimensional manifold assumption.
To the best of our knowledge, our work is the first to show that deep nonparametric regression can mitigate the \text{curse of dimensionality} under an \text{approximate manifold} assumption.
Moreover, we have provided  a characterization of  how excess risk bounds depend on the network architecture,  obtained a new error bound with a new proof under the Minkowski dimension assumption and established a new error bound with the optimal convergence rate and an improved prefactor under the exact manifold assumption.

As we have remarked below Theorem \ref{thm_apx},
our work builds on the results of \cite{shen2019deep} and
 \cite{lu2020deep}.
Specifically, \cite{shen2019deep} derived a quantitative and non-asymptotic approximation rate  $19\sqrt{d}\omega_f(N^{-2/d}L^{-2/d})$ in terms of width $\mathcal{O}(N)$ and depth $\mathcal{O}(L)$ of the ReLU networks for continuous target function $f$,  where $\omega_f(\cdot)$ denotes its modulus of continuity. When this result is applied to H\"older continuous target functions with order (or smoothness index) $\alpha\in(0,1]$, the approximation rate becomes $19\sqrt{d}N^{-2\alpha/d}L^{-2\alpha/d}$, which is nearly optimal. \cite{lu2020deep} showed that deep ReLU networks of width $\mathcal{O}(N\log N)$ and depth $\mathcal{O}(L\log L)$ can approximate smooth function $f\in C^s([0,1]^d)$ with a nearly optimal (up to a logarithmic factor) approximation error $85(s+1)^d8^s\Vert f\Vert_{C^s([0,1]^d)} n^{-2s/d}L^{-2s/d}$, where $C^s([0,1]^d)$ denotes smooth function space with smoothness index $s\in N^+$(a positive integer), and $\Vert\cdot\Vert_{C^s([0,1]^d)}$ denotes the H\"older norm. The result holds for smooth target function with its smoothness index being a positive integer $s\ge1$, while the prefactor of the approximation error bound is  $(s+1)^d$, which depends on the dimension $d$ exponentially.
In comparison, our approximation results hold for H\"older smooth target functions with smoothness index $\beta>0$.
Moreover, when the smoothness index $\beta>1$, our approximation error bound
has a prefactor depending on $d$ polynomially.

There are several
limitations in this work.
{\color{black} First,  the optimal rate of convergence under the approximate manifold assumption remains unknown to us.
It appears that one is unlikely to obtain an error bound with rate depending only on the intrinsic dimension $d_\mathcal{M}$ of the manifold, as the dimension of  an approximate manifold is still $d$.
Second, it is not clear what are the best prefactors for the error bounds in the present setting.
This is an interesting by challenging problem in the present setting.
}
Finally,
it would be interesting to generalize the results in this work
to other problems,
 such as density estimation, conditional density estimation and generative learning.
These problems deserve further study in the future.

\section*{Acknowledgements}
The authors wish to thank the editors, the associate editor and three anonymous reviewers for
their insightful comments and constructive suggestions that helped improve the paper significantly.
We are especially grateful to them for their suggestions to
consider ReLU network approximation for higher order H\"older smooth functions, the generalization error bound under an exact manifold assumption  and when data is supported on a set with a low Minkowski dimension,  which led to Theorems \ref{thm_apx}, \ref{thm5} and \ref{thm4}.

Y. Jiao is supported  by the National Science Foundation of China grant 11871474 and by the research fund of KLATASDSMOE of China.
Y. Lin is supported by the Hong Kong Research Grants Council (Grant No.
14306219 and 14306620), the National Natural Science Foundation of China (Grant No.
11961028) and Direct Grants for Research, The Chinese University of Hong Kong.
J. Huang is partially supported by the U.S. NSF grant DMS-1916199 while he was at the University of Iowa and the research grant P0042888 from The Hong Kong Polytechnic University.
	

\begin{appendix}

\setcounter{equation}{0}  
\renewcommand{\theequation}{A.\arabic{equation}}
\setcounter{table}{0}
\renewcommand{\thetable}{A.\arabic{table}}
\setcounter{figure}{0}
\renewcommand{\thefigure}{A.\arabic{figure}}
\setcounter{definition}{0}
\renewcommand{\thedefinition}{A.\arabic{definition}}
\setcounter{equation}{0}  
\setcounter{theorem}{0}
     \renewcommand{\thetheorem}{\Alph{section}.\arabic{theorem}}

\section{Proofs}
In this appendix, we prove Lemmas \ref{lemma0} and \ref{lemma1}, Theorems  \ref{thm_apx}, \ref{thm2},  \ref{thm3}, \ref{thm5} and \ref{thm4}, {\color{black} Corollaries \ref{thm_apx1} and \ref{c1}}.
Theorem {\ref{thm1}}  is a direct consequence of Lemma \ref{lemma1} and Theorem 1 in \cite{yarotsky2018optimal}, thus we omit its proof.

\subsection{Proof of Lemma \ref{lemma0}}
%

\begin{proof}
Since $f_0$ is the minimizer of quadratic functional  $L(\cdot)$,   by direct calculation
we have
\begin{equation}\label{ine1}
\mathbb{E}_{S} [\|\hat{f}_n-f_0\|_{L^2(\nu)}^2]=\mathbb{E}_{S} [L(\hat{f}_n)-L(f_0)].
\end{equation}
By the definition of the empirical risk minimizer, we have
$$L_n(\hat{f}_n)-L_n(f_0)\leq L_n(\bar{f})-L_n(f_0),$$
 where $\bar{f} \in\arg\inf_{f\in\mathcal{F}_n}\Vert f-f_0\Vert^2_{L^2(\nu)}.$
Taking expectations on both side we get
\begin{equation}\label{ine2}
\mathbb{E}_{S}[L_n(\hat{f}_n)-L(f_0)]\leq L(\bar{f})-L(f_0) = \|\bar{f}-f_0\|^2_{L^2(\nu)}.
\end{equation}
Multiplying  both sides of (\ref{ine2}) by 2,  adding the resulting inequality  with (\ref{ine1}) and rearranging the terms, we obtain Lemma \ref{lemma0}.
	\end{proof}

\subsection{Proof of Lemma \ref{lemma1}}
\begin{proof}
Let $S=\{Z_i=(X_i,Y_i)\}_{i=1}^n$ be a random sample form the distribution of $Z=(X,Y)$ and $S^\prime=\{Z_i^\prime=(X^\prime_i,Y^\prime_i)\}_{i=1}^n$ be another sample independent of $S$. Define $g(f,Z_i)=(f(X_i)-Y_i)^2-(f_0(X_i)-Y_i)^2$ for any $f$ and sample $Z_i$.
Observing
	\begin{equation}\label{bound0}
		\mathbb{E}_{S}[ L(f_0)-2L_{n}(\hat{f}_n)+L(\hat{f}_n)]=\mathbb{E}_{S}\Big[ \frac{1}{n}\sum_{i=1}^n\big\{-2g(\hat{f}_\phi,Z_i)+\mathbb{E}_{S^\prime}g(\hat{f}_\phi,Z_i^\prime)\big\}\Big].
\end{equation}
By Lemma \ref{lemma0} and the above display, it is seen that the expected  prediction error  $$\mathcal{R}(\hat{f}_n):=\mathbb{E}_{S}[\|\bar{f}-f_0\|^2_{L^2(\nu)}]$$
is upper bounded by the sum of the expectation of a stochastic term and the approximation error.
	Next, we bound the expectation of  the stochastic term with truncation and the classical chaining technique from the empirical process theory. In the following, for ease of presentation, we write {$G(f,Z_i)=\mathbb{E}_{S^\prime}\{g(f,Z_i^\prime)\}-2g(f,Z_i)$ for $f\in\mathcal{F}_\phi$}.

	Let $\beta_n\geq \mathcal{B}\geq1$ be a positive number which may depend on the sample size $n$. Denote $T_{\beta_n}$ as the truncation operator at level $\beta_n$, i.e., for any $Y\in\mathbb{R}$,
$T_{\beta_n}Y=Y$ if $\vert Y\vert\leq\beta_n$ and $T_{\beta_n}Y= \beta_n\cdot {\rm sign}(Y)$ otherwise. Let $f_{\beta_n}(x)=\mathbb{E}\{T_{\beta_n}Y|X=x\}$ be the regression function of the truncated $Y$. Recall that $g(f,Z_i)=(f(X_i)-Y_i)^2-(f_0(X_i)-Y_i)^2$, we define $g_{\beta_n}(f,Z_i)=(f(X_i)-T_{\beta_n}Y_i)^2-(f_{\beta_n}(X_i)-T_{\beta_n}Y_i)^2$ and $G_{\beta_n}(f,Z_i)=\mathbb{E}_{S^\prime}\{g_{\beta_n}(f,Z_i^\prime)\}-2g_{\beta_n}(f,Z_i)$ for $f\in\mathcal{F}_n$.
Then for any $f\in\mathcal{F}_n$ we have
{	\begin{align*}
	 \Big\vert g(f,Z_i)-g_{\beta_n}(f,Z_i)\Big\vert=& \Big\vert 2\{f(X_i)-f_0(X_i)\}(T_{\beta_n}Y_i-Y_i)\\
		&+(f_{\beta_n}(X_i)-T_{\beta_n}Y_i)^2-(f_0(X_i)-T_{\beta_n}Y_i)^2\Big\vert\\
		\leq & \Big\vert 2\{f(X_i)-f_0(X_i)\}(T_{\beta_n}Y_i-Y_i)\Big\vert \\
		&+\Big\vert(f_{\beta_n}(X_i)-T_{\beta_n}Y_i)^2-(f_0(X_i)-T_{\beta_n}Y_i)^2\Big\vert\\
		\leq & 4\mathcal{B}\vert T_{\beta_n}Y_i-Y_i\vert\\
		&+\vert f_{\beta_n}(X_i)-f_0(X_i)\vert\vert f_{\beta_n}(X_i)+f_0(X_i)-2T_{\beta_n}Y_i\vert\\
		\leq &4\mathcal{B}\vert Y_i\vert I(\vert Y_i\vert>\beta_n) +4\beta_n\vert f_{\beta_n}(X_i)-f_0(X_i)\vert\\
		\leq &4\mathcal{B}\vert Y_i\vert I(\vert Y_i\vert>\beta_n) +4\beta_n\vert T_{\beta_n}Y_i-Y_i\vert\\
		\leq  &4\mathcal{B}\vert Y_i\vert I(\vert Y_i\vert>\beta_n) +4\beta_n\vert Y_i\vert I(\vert Y_i\vert>\beta_n),
\end{align*}
}
and
\begin{align*}
		\mathbb{E}_S\{ g(f,Z_i)\}  \leq& \mathbb{E}_S\{ g_{\beta_n}(f,Z_i)\} +4\mathcal{B}\mathbb{E}_S\{\vert Y_i\vert I(\vert Y_i\vert>\beta_n)\} +4\beta_n\mathbb{E}_S\{\vert Y_i\vert I(\vert Y_i\vert>\beta_n)\}\\
		\leq & \mathbb{E}_S\{ g_{\beta_n}(f,Z_i)\} + 8\beta_n\frac{2}{\sigma_Y}\mathbb{E}_S\Big[\frac{\sigma_Y}{2}\vert Y_i\vert\exp\big\{\frac{\sigma_Y}{2}(\vert Y_i\vert-\beta_n)\big\}\Big]\\
		\leq & \mathbb{E}_S\{ g_{\beta_n}(f,Z_i)\} + 16\frac{\beta_n}{\sigma_Y}\mathbb{E}_S\exp(\sigma_Y\vert Y_i\vert)\exp(-{\sigma_Y}\beta_n/2).
	\end{align*}
	By Assumption 2, the response $Y$ is sub-exponentially distributed and $\mathbb{E}\exp(\sigma_Y\vert Y_i\vert)<\infty$. Therefore,
	\begin{equation} \label{bound2}
		\mathbb{E}_{S}\Big[ \frac{1}{n}\sum_{i=1}^n G(\hat{f}_n,Z_i) \Big]
		\leq\mathbb{E}_{S}\Big[ \frac{1}{n}\sum_{i=1}^n G_{\beta_n}(\hat{f}_n,Z_i) \Big]+c_1\beta_n\exp(-\sigma_Y\beta_n/2),
	\end{equation}
where $c_1$ is a constant not depending on $n$ and $\beta_n$.

{\color{black}
Note that $\vert T_{\beta_n}Y\vert\le \beta_n$, $\Vert g_{\beta_n} \Vert_\infty\le\beta_n$ and $\beta_n\ge\mathcal{B}\ge1$. 
Then by Theorem 11.4 of \cite{gyorfi2006distribution}, for each $n\ge1$,
	\begin{align*}
		&P\Big\{\frac{1}{n}\sum_{i=1}^nG_{\beta_n}(\hat{f}_n,Z_i)>t\Big\}\\
		\le&P\Big\{ \exists f\in\mathcal{F}_n: \frac{1}{n}\sum_{i=1}^nG_{\beta_n}(f_,Z_i)>t\Big\}\\
		=&P\Big\{\exists f\in\mathcal{F}_n: \mathbb{E}_{S^\prime} \{g_{\beta_n}(f,Z_i^\prime)\}-\frac{2}{n}\sum_{i=1}^ng_{\beta_n}(f,Z_i)>t\Big\}\\
		\le& 14\mathcal{N}_{2n}(\frac{t}{80\beta_n},\Vert\cdot\Vert_\infty,\mathcal{F}_n)\exp\left(-\frac{tn}{5136\beta_n^4}\right)
	\end{align*}
	This leads to a tail probability bound of $\sum_{i=1}^n G_{\beta_n}(f_{j^*},Z_i)/n$. 
	
	Then for $a_n>0$,
	\begin{align*}
		\mathbb{E}_S\Big[ \frac{1}{n}\sum_{i=1}^nG_{\beta_n}(f_{j^*},Z_i)\Big]\leq& a_n +\int_{a_n}^\infty P\Big\{\frac{1}{n}\sum_{i=1}^nG_{\beta_n}(f_{j^*},Z_i)>t\Big\} dt\\
		\leq& a_n+\int_{a_n}^\infty14\mathcal{N}_{2n}(\frac{t}{80\beta_n},\Vert\cdot\Vert_\infty,\mathcal{F}_n)\exp\left(-\frac{tn}{5136\beta_n^4}\right) dt\\
			\leq& a_n+\int_{a_n}^\infty14\mathcal{N}_{2n}(\frac{a_n}{80\beta_n},\Vert\cdot\Vert_\infty,\mathcal{F}_n)\exp\left(-\frac{tn}{5136\beta_n^4}\right) dt\\
		\leq& a_n+ 14\mathcal{N}_{2	n}(\frac{a_n}{80\beta_n},\Vert\cdot\Vert_\infty,\mathcal{F}_n)\exp\left(-\frac{a_nn}{5136\beta_n^4}\right)\frac{5136\beta_n^4}{n}.
	\end{align*}
We choose $a_n=\log(14\mathcal{N}_{2n}(\frac{1}{n},\Vert\cdot\Vert_\infty,\mathcal{F}_n)
)\cdot{5136\beta_n^4}/{n}$. Note that  $a_n/(80\beta_n)\ge 1/n.$ and $\mathcal{N}_{2n}(\frac{1}{n},\Vert\cdot\Vert_\infty,\mathcal{F}_n)\ge \mathcal{N}_{2n}(\frac{a_n}{80\beta_n},\Vert\cdot\Vert_\infty,\mathcal{F}_n)$.
	Then we have
	\begin{equation} \label{bound3}
		\mathbb{E}_S\Big[ \frac{1}{n}\sum_{i=1}^nG_{\beta_n}(f_{j^*},Z_i)\Big]\leq \frac{5136\beta_n^4(\log(14\mathcal{N}_{2n}(\frac{1}{n},\Vert\cdot\Vert_\infty,\mathcal{F}_n))+1)}{n}.
	\end{equation}
	Setting 
	$\beta_n=c_2\mathcal{B}\log n$ and combining 
	(\ref{bound2}) and  (\ref{bound3}), we prove (\ref{bound5a}). Further combining  (\ref{bound0}) we get

	\begin{equation} \label{bound5}
		\mathcal{R}(\hat{f}_n)\leq c_3\mathcal{B}^4\frac{\log\mathcal{N}_{2n}(\frac{1}{n},\Vert \cdot\Vert_\infty,\mathcal{F}_n)(\log n)^4}{n}+ 2\Vert f^*_n-f_0\Vert^2_{L^2(\nu)},
	\end{equation}
	where $c_3>0$ is a constant not depending on $n$ or $\mathcal{B}$.

Lastly, we will give an upper bound on the  covering number by the VC dimension of $\mathcal{F}_n$ through its parameters. Denote the pseudo dimension of $\mathcal{F}_n$ by ${\rm Pdim}(\mathcal{F}_n)$,
by Theorem 12.2 in \cite{anthony1999}, for $2n\geq {\rm Pdim}(\mathcal{F}_n)$,
$$\mathcal{N}_{2n}(\frac{1}{n},\Vert \cdot\Vert_\infty,\mathcal{F}_n)\leq\Big(\frac{4e\mathcal{B}n^2}{{\rm Pdim}(\mathcal{F}_n)}\Big)^{{\rm Pdim}(\mathcal{F}_n)}.$$
Moreover, based on Theorem 3 and 6 in \cite{bartlett2019nearly}, there exist universal constants $c$, $C$ such that
	$$c\cdot\mathcal{S}\mathcal{D}\log(\mathcal{S}/\mathcal{D})\leq{\rm Pdim}(\mathcal{F}_n)\leq C\cdot\mathcal{S}\mathcal{D}\log(\mathcal{S}).$$
	Combining the upper bound of the covering number and the pseudo dimension with (\ref{bound5}), we have
	\begin{equation} \label{bound6}
		\mathcal{R}(\hat{f}_n)\leq c_4\mathcal{B}^5\frac{\mathcal{S}\mathcal{D}\log(\mathcal{S})(\log n)^5}{n}+ 2\Vert f^*_n-f_0\Vert^2_{L^2(\nu)},
	\end{equation}
	for some constant $c_4>0$  not depending on $n$, $d$, $\mathcal{B}$, $\mathcal{S}$ or $\mathcal{D}$. Therefore,  (\ref{oracle}) follows.
This completes the proof of Lemma \ref{lemma1}.

}
\end{proof}

\subsection{Proof of Theorem \ref{thm_apx}}
	This approximation result  improves the prefactor in $d$ of the network width in Theorem 2.2 in \cite{lu2020deep}. 	The main idea of our proof is to approximate the Taylor expansion of H\"older smooth $f$. By Lemma A.8 in \cite{petersen2018optimal}, for any $x,x_0\in[0,1]^d$, we have
	$$\Big\vert f(x)-\sum_{\Vert\alpha\Vert_1\le s}\frac{\partial^\alpha f(x_0)}{\alpha!}(x-x_0)^\alpha\Big\Vert\leq d^s\Vert x-x_0\Vert^\beta_2.$$
	This reminder term could be well controlled when the approximation to Taylor expansion in implemented in a fairly small local region.
	Then we can focus on the approximation of the Taylor expansion locally.
The proof is divided into three parts:
	\begin{itemize}
		\item Partition $[0,1]^d$ into small cubes $\bigcup_\theta Q_\theta$, and construct a network $\psi$ that approximately maps each $x\in Q_\theta$ to a fixed point $x_\theta\in Q_\theta$. Hence, $\psi$ approximately discretize $[0,1]^d$.
		\item For any multi-index $\alpha$,  construct a network $\phi_\alpha$ that approximates the Taylor coefficient $x\in Q_\theta \mapsto \partial^\alpha f(\psi(x_\theta))$. Once $[0,1]^d$ is discretized, the approximation is reduced to a data fitting problem.
		\item Construct a network $P_\alpha(x)$ to approximate the polynomial $x^\alpha:=x_1^{\alpha_1}\ldots x_d^{\alpha_d}$ where $x=(x_1,\ldots,x_d)^\top\in\mathbb{R}^d$ and $\alpha=(\alpha_1,\ldots,\alpha_d)^\top\in\mathbb{N}^d_0$. In particular, we can construct a network $\phi_\times(\cdot,\cdot)$ approximating the product function of two scalar inputs.
	\end{itemize}
	Then our construction of neural network can be written in the form,
	$$\phi(x)=\sum_{\Vert\alpha\Vert_1\le s}\phi_\times\Big( \frac{\phi_\alpha(x)}{\alpha!},P_\alpha(x-\psi(x))\Big).$$

	\begin{proof}
		Without loss of generality, we assume the H\"older norm of $f$ is 1, i.e. $f\in\mathcal{H}^\beta([0,1]^d,1)$. The reason is that we can always approximate $f/B_0$ firstly by a network $\phi$ with approximation error $\epsilon$, then the scaled network $B_0\phi$ will approximate $f$ with error no more than $\epsilon B_0$. Besides, it is a trivial case when the H\"older norm of $f$ is 0.
		Firstly, when $\beta>1$, we divide the proof into three steps as follows.
		
		\textbf{Step 1:} Discretization.\\
		Given $K\in\mathbb{N}^+$ and $\delta\in(0,1/(3K)]$, for each $\theta=(\theta_1,\ldots,\theta_d)\in\{0,1,\ldots,K-1\}^d$, we define
		$$Q_\theta:=\Big\{x=(x_1,\ldots,x_d):x_i\in[\frac{\theta_i}{K},\frac{\theta_i+1}{K}-\delta\cdot1_{\theta_i<K-1}],i=1,\ldots,d\Big\}.$$
		Note that $[0,1]^d\backslash\Omega([0,1]^d,K,\delta)=\bigcup_{\theta}Q_\theta$. By the definition of $Q_\theta$, the region $[0,1]^d$ is approximately divided into hypercubes. By Lemma \ref{lemmab1}, there exists a ReLU network $\psi_1$ with width $4\lfloor N^{1/d}\rfloor+3$ and depth $4M+5$ such that
		$$\psi_1(x)=\frac{k}{K},\quad {\rm if\ } x\in[\frac{k}{K},\frac{k+1}{K}-\delta\cdot 1_{\{k<K-1\}}], k=0,1,\ldots,K-1.$$
We define
		$$\psi(x):=(\psi_1(x_1),\ldots,\psi_1(x_d)),\quad x=(x_1,\ldots,x_d)\in\mathbb{R}^d.$$
		Then we have $\psi(x)={\theta}/{K}:=({\theta_1}/{K},\ldots,{\theta_d}/{K})^\top$ for $x\in Q_\theta$ and $\psi$ is a ReLU network with width $d(4\lfloor N^{1/d}\rfloor+3)$ and depth $4M+5$.
		
		\textbf{Step 2:} Approximation of Taylor coefficients.\\
		Since $\theta\in\{0,1,\ldots,K-1\}^d$ is one-to-one correspondence to $i_\theta:=\sum_{j=1}^d\theta_j K^{j-1}\in\{0,1\ldots,K^d-1\}$, we define
		$$\psi_0(x):=(K,K^2,\ldots,K^d)\cdot\psi(x)=\sum_{j=1}^d\psi_1(x_j)K^j,\quad x\in\mathbb{R}^d,$$
		then
		$$\psi_0(x)=\sum_{j=1}^d\theta_jK^{j-1}=i_\theta,\quad {\rm if\ }x\in Q_\theta, \theta\in\{0,1,\ldots,K-1\}^d,$$
		where $\psi_0(x)$ has width $d(4\lfloor N^{1/d}\rfloor+3)$ and depth $4M+5$.
		For any $\alpha\in\mathbb{N}^d_0$ satisfying $\Vert\alpha\Vert_1\leq s$ and each $i=i_\theta\in\{0,1,\ldots,K^d-1\}$, we denote $\xi_{\alpha,i}:=(\partial^\alpha f(\theta/K)+1)/2\in[0,1].$ Since $K^d\leq N^2M^2$, by Lemma \ref{lemmab2}, there exists a ReLU network $\varphi_\alpha$ with width $16(s+1)(N+1)\lceil \log_2(8N)\rceil$ and depth $5(M+2)\lceil\log_2(4M)\rceil$ such that
		$$\vert \varphi_\alpha(i)-\xi_{\alpha,i}\vert\leq(NM)^{-2(s+1)},$$
		for all $i\in\{0,1,\ldots,K^d-1\}.$
		We define
		$$\phi_\alpha(x):=2\varphi_\alpha(\psi_0(x))-1\in[-1,1],\quad x\in\mathbb{R}^d.$$
		Then $\phi_\alpha$ can be implemented by a network with width $16d(s+1)(N+1)\lceil \log_2(8N)\rceil\leq 32d(s+1)N\lceil \log_2(8N)\rceil$ and depth $5(M+2)\lceil\log_2(4M)\rceil+4M+5\leq15M\lceil\log_2(8M)\rceil$. And we have for any $\theta\{0,1,\ldots,K-1\}^d$, if $x\in Q_\theta$,
		\begin{equation} \label{apx1}
			\vert \phi_\alpha(x)-\partial^\alpha f(\theta/K)\vert=2\vert \varphi_\alpha(i_\theta)-\xi_{\alpha,i_\theta}\vert\leq2(NM)^{-2(s+1)}.
		\end{equation}
		
		\textbf{Step 3:} Approximation of $f$ on $\bigcup_{\theta\in\{0,1,\ldots,K-1\}^d} Q_\theta$.\\
		Let $\varphi(t)=\min\{\max\{t,0\},1\}=\sigma(t)-\sigma(t-1)$ for $t\in\mathbb{R}$ where $\sigma(\cdot)$ is the ReLU activation function. With a slightly abuse of the notation, we extend its definition to $\mathbb{R}^d$ coordinate-wisely, i.e., $\varphi:\mathbb{R}^d\to[0,1]^d$ and $\varphi(x)=x$ for any $x\in[0,1]^d$. By Lemma \ref{lemmab3}, there exists a ReLU network with width $9N+1$ and depth $2(s+1)M$ such that for any $t_1,t_2\in[-1,1]$,
		\begin{gather}
			\label{apx2}
			\vert t_1t_2-\phi_\times(t_1,t_2)\vert\leq 24N^{-2(s+1)M},
		\end{gather}
		By Lemma \ref{lemmab4}, for any $\alpha\in\mathbb{N}^d_0$ with $ \alpha\Vert_2\leq s$, there exists a ReLU network $P_\alpha$ with width $9N+s+8$ and depth $7(s+1)^2M$ such that
$P_\alpha(x)\in[-1,1]$ and
		\begin{gather}
			\label{apx4}
			\vert P_\alpha(x)-x^\alpha\vert\leq9(s+1)(N+1)^{-7(s+1)M}.
		\end{gather}
		
		For any $x\in Q_\theta$, $\theta\in\{0,1,\ldots,K-1\}^d$, we can now approximate the Taylor expansion of $f(x)$ by combined sub-networks. Thanks to Lemma A.8 in \cite{petersen2018optimal}, we have the following error control for $x\in Q_\theta$,
		\begin{equation}\label{apx6}
			\Big\vert f(x)-f(\frac{\theta}{K})-\sum_{1\leq\Vert\alpha\Vert_1\leq s}\frac{\partial^\alpha f(\frac{\theta}{K})}{\alpha!}(x-\frac{\theta}{K})^\alpha\Big\vert\leq d^s\Vert x-\frac{\theta}{K}\Vert_2^\beta\leq d^{s+\beta/2}K^{-\beta}.
		\end{equation}
		Motivated by this, we define
		\begin{align*}
			\tilde{\phi}_0(x)&:=\phi_{\textbf{0}_d}(x)+\sum_{1\leq \Vert\alpha\Vert_1\leq s} \phi_\times\Big( \frac{\phi_\alpha(x)}{\alpha!},P_\alpha(\varphi(x)-\phi(x))\Big),\\
			{\phi}_0(x)&:=\sigma(\tilde{\phi}_0(x)+1)-\sigma(\tilde{\phi}_0(x)-1)-1\in[-1,1],
		\end{align*}
		where $\textbf{0}_d=(0,\ldots,0)\in\mathbb{N}^d_0$. Observe that the number of terms in the summation can be bounded by
		$$\sum_{\alpha\in\mathbb{N}^d_0,\Vert\alpha\Vert_1\le s}1=\sum_{j=0}^s\sum_{\alpha\in\mathbb{N}^d_0,\Vert\alpha\Vert_1=j}1\leq\sum_{j=0}^sd^s\leq(s+1)d^s.$$
		Recall that width and depth of $\varphi$ is $(2d,1)$, width and depth of $\psi$ is $(d(4\lfloor N^{1/d}\rfloor+3),4M+5)$, width and depth of $P_\alpha$ is $(9N+s+8,7(s+1)^2M)$, width and depth of $\phi_\alpha$ is width ($16d(s+1)(N+1)\lceil \log_2(8N)\rceil$,$5(M+2)\lceil\log_2(4M)\rceil+4M+5$) and width and depth of $\phi_\times$ is $(9N+1,2(s+1)M)$. Hence, by our construction, $\phi_0$ can be implemented by a neural network with width $38(s+1)^2d^{s+1}N\lceil\log_2(8N)\rceil$ and depth $21(s+1)^2M\lceil\log_2(8M)\rceil$. The approximation error $\vert f(x)-\phi_0(x)\vert$ can be bounded as follows. For any $x\in Q_\theta$, $\varphi(x)=x$ and $\psi(x)=\theta/K$. Then by the triangle inequality and (\ref{apx6}),
		\begin{align*}
			&\vert f(x)-\phi_0(x)\vert\leq \vert f(x)-\tilde{\phi}_0(x)\vert\\
			\leq& \vert f(\theta/K)- \phi_{\textbf{0}_d}(x)\vert+d^{s+\beta/2}K^{-\beta}\\
			&+\sum_{1\leq\Vert\alpha\Vert_1\leq s}\Big\vert \frac{\partial^\alpha f(\theta/K)}{\alpha!}(x-\theta/K)^\alpha-\phi_\times\big( \frac{\phi_\alpha(x)}{\alpha!},P_\alpha(x-\theta/K)\big)\Big\vert\\
			=&d^{s+\beta/2}\lfloor (MN)^{2/d}\rfloor^{-\beta}+\sum_{\Vert\alpha\Vert_1\leq s}\mathcal{E}_\alpha,
		\end{align*}
		where we denote $\mathcal{E}_\alpha=\Big\vert \frac{\partial^\alpha f(\theta/K)}{\alpha!}(x-\theta/K)^\alpha-\phi_\times\big( \frac{\phi_\alpha(x)}{\alpha!},P_\alpha(x-\theta/K)\big)\Big\vert$ for each $\alpha\in\mathbb{N}^d_0$ with $\Vert\alpha\Vert_1\leq s$.
		Using the inequality $\vert t_1t_2-\phi_\times(t_3,t_4)\vert\leq\vert t_1t_2-t_3t_2\vert+\vert t_3t_2-t_3t_4\vert+\vert t_3t_4-\phi_\times(t_3,t_4)\vert\leq \vert t_1-t_3\vert+\vert t_2-t_4\vert+\vert t_3t_4-\phi_\times(t_3,t_4)\vert$ for any $t_1,t_2,t_3,t_4\in[-1,1]$, and by (\ref{apx1}), (\ref{apx2}) and (\ref{apx4}), for $1\leq\Vert\alpha\Vert_1\leq s$ we have
		\begin{align*}
			\mathcal{E}_\alpha\leq&\frac{1}{\alpha!}\vert\partial^\alpha f(\theta/K)-\phi_\alpha(x)\vert+\vert(x-\theta/K)^\alpha-P_\alpha(x-\theta/K)\vert\\
			&+\vert\frac{\phi_\alpha(x)}{\alpha!}P_\alpha(x-\theta/K)-\phi_\times\Big(\frac{\phi_\alpha(x)}{\alpha!},P_\alpha(x-\theta/K)\Big)\\
			\leq&2(NM)^{-2(s+1)}+9(s+1)(N+1)^{-7(s+1)M}+6N^{-2(s+1)M}\\
			\leq&(9s+17)(NM)^{-2(s+1)}.
		\end{align*}
		It is easy to check that the bound is also true when $\Vert \alpha\Vert_1=0$ and $s=0$.  Therefore,
		\begin{align*}
			\vert f(x)-\phi_0(x)\vert &\leq \sum_{1\leq \Vert\alpha\Vert_1\leq s} (9s+17)(NM)^{-2(s+1)}+d^{s+\beta/2}(NM)^{-2\beta/d}\\
			&\leq (s+1)d^s(9s+17)(NM)^{-2(s+1)}+d^{s+\beta/2}(NM)^{-2\beta/d}\\
			&\leq 18(s+1)^2d^{s+\beta/2}(NM)^{-2\beta/d},
		\end{align*}
		for any $x\in\bigcup_{\theta\in\{0,1,\ldots,K-1\}^d}Q_\theta$. And for $f\in\mathcal{H}^\beta([0,1]^d,B_0)$, by approximate $f/B_0$ firstly, we know there exists a function implemented by a neural network with the same width and depth as  $\phi_0$, such that
		$$\vert f(x)-\phi_0(x)\vert\leq 18B_0(s+1)^2d^{s+\beta/2}(NM)^{-2\beta/d},$$
		for any $x\in\bigcup_{\theta\in\{0,1,\ldots,K-1\}^d}Q_\theta$.
		
		Lastly, when $0<\beta\le1$, $f$ is a H\"older continuous function with order $\beta$ and constant H\"older $B_0$,
		then by Theorem 1.1 in \cite{shen2019deep}, there exists a function $\phi_0$ which is implemented by a neural network with width $\max\{4d\lfloor N^{1/d}\rfloor+3d,12N+8\}$ and depth $12M+14$, such that
		$$\vert f(x)-\phi_0(x)\vert\leq 18\sqrt{d}B_0(NM)^{-2\beta/d},$$
		for any $x\in\bigcup_{\theta\in\{0,1,\ldots,K-1\}^d}Q_\theta$.
		Combining the results for $\beta\in(0,1]$ and $\beta>1$, we have for $f\in\mathcal{H}^\beta([0,1]^d,B_0)$, there exists a function $\phi_0$ implemented by a neural network with width $38(s+1)^2d^{s+1}N\lceil\log_2(8N)\rceil$ and depth $21(s+1)^2M\lceil\log_2(8M)\rceil$ such that
			$$\vert f(x)-\phi_0(x)\vert\leq 18B_0(s+1)^2d^{s+\beta\vee1/2}(NM)^{-2\beta/d},$$
		for any $x\in\bigcup_{\theta\in\{0,1,\ldots,K-1\}^d}Q_\theta$ where $s=\lfloor\beta\rfloor$.
		
	\end{proof}

\subsection{Proof of Corollary \ref{thm_apx1}}
	
	  We prove 
Corollary \ref{thm_apx1} based on Theorem \ref{thm_apx}.
	  \begin{proof}
	  	Let $\mathcal{E}=18B_0(s+1)^2d^{s+\beta/2}(NM)^{-2\beta/d}$. We construct a neural network $\phi$ that uniformly approximates $f$ on $[0,1]^d$. To present the construction, we denote ${\rm mid}(t_1,t_2,t_3)$ as the function that returns the middle value of three inputs $t_1,t_2,t_3\in\mathbb{R}$. It is easy to check that
	  	$$\max\{t_1,t_2\}=\frac{1}{2}(\sigma(t_1+t_2)-\sigma(-t_1-t_2)+\sigma(t_1-t_2)+\sigma(t_2-t_1)).$$
	  	Thus $\max\{t_1,t_2,t_3\}=\max\{\max\{t_1,t_2\},\sigma(t_3)-\sigma(-t_3)\}$ can be implemented by a ReLU network
	  	with width 6 and depth 2. Similar construction holds for $\min\{t_1,t_2,t_3\}$. Since
	  	$${\rm mid}(t_1,t_2,t_3)=\sigma(t_1+t_2+t_3)-\sigma(-t_1-t_2-t_3)-\max\{t_1,t_2,t_3\}-\min\{t_1,t_2,t_3\},$$
	  	the function ${\rm mid}(\cdot,\cdot,\cdot)$ can be implemented by a ReLU network with width 14 and depth 2.
	  	Let $\{e_i\}_{i=1}^d$ be the standard orthogonal basis in $\mathbb{R}^d$, we inductively define
	  	$$\phi_i(x):={\rm mid}(\phi_{i-1}(x-\delta e_i),\phi_{i-1}(x),\phi_{i-1}(x+\delta e_i))\in[-1,1],\ i=1,\ldots,d,$$
	  	where $\phi_0$ is defined in the proof of Theorem \ref{thm_apx}.
	  	Then $\phi_d$ can be implemented by a ReLU network with width $38(s+1)^23^dd^{s+1}N\lceil\log_2(8N)\rceil$ and depth $21(s+1)^2M\lceil\log_2(8M)\rceil+2d$ recalling that $\phi_0$ has width $38(s+1)^2d^{s+1}N\lceil\log_2(8N)\rceil$ and depth $21(s+1)^2M\lceil\log_2(8M)\rceil$. Denote $Q(K,\delta):=\bigcup_{k=0}^{K-1}[\frac{k}{K},\frac{k+1}{K}-\delta\cdot 1_{k<K-1}]$ and define
	  	$$E_i:=\{(x_1,\ldots,x_d)\in[0,1]^d:x_j\in Q(K,\delta),j>i\},$$
	  	for $i=0,\ldots,d$. Then $E_0=\bigcup_{\theta\in\{0,1,\ldots,K-1\}^d} Q_\theta$ and $E_d=[0,1]^d$. We assert that
	  	$$\vert\phi_i(x)-f(x)\vert\le \mathcal{E}+iB_0\delta^{\beta\wedge1},\ \forall x\in E_i, i=0,\ldots,d,$$
	  	 where $a\wedge b:=\min\{a,b\}$ for $a,b\in\mathbb{R}$.
	  	
	  	 We prove the assertion by induction. Firstly, it is true for $i = 0$ by construction. Assume the assertion is true for some $i$, we will prove that it is also holds for $i + 1$. Note that for any $x\in E_{i+1}$, at least two of $x-\delta e_{i+1}$, $x$ and $x+ \delta e_{i+1}$ are in $E_i$. Therefore,  by assumption and the inequality $\vert f(x)-f(x\pm\delta e_{i+1})\vert\leq B_0\delta^{\beta\wedge1}$, at least two of the following inequalities hold,
	  	 \begin{align*}
	  	 	\vert \phi_i(x-\delta e_{i+1})-f(x)\vert\leq &\vert \phi_i(x-\delta e_{i+1})-f(x-\delta e_{i+1})\vert +B_0\delta^{\beta\wedge1}\leq \mathcal{E}+(i+1)B_0\delta^{\beta\wedge1},\\
	  	 		\vert \phi_i(x)-f(x)\vert\leq & \mathcal{E}+iB_0\delta^{\beta\wedge1},\\
	  	 		\vert \phi_i(x+\delta e_{i+1})-f(x)\vert\leq & \vert \phi_i(x+\delta e_{i+1})-f(x+\delta e_{i+1})\vert +B_0\delta^{\beta\wedge1}\leq \mathcal{E}+(i+1)B_0\delta^{\beta\wedge1}.
	  	  	 \end{align*}
	  	 In other words, at least two of $\phi_i(x-\delta e_{i+1})$, $\phi_i(x)$ and $\phi_i(x+\delta e_{i+1})$ are in the interval $[f(x)-\mathcal{E}-(i+1)B_0\delta^{\beta\wedge1},f(x)+\mathcal{E}+(i+1)B_0\delta^{\beta\wedge1}].$ Hence, their middle value $\phi_{i+1}(x)={\rm mid}(\phi_i(x-\delta e_{i+1},\phi_i(x),\phi_i(x+\delta e_{i+1})))$ must be in the same interval, which means
	  	 $$\vert \phi_{i+1}(x)-f(x)\vert\leq\mathcal{E}+(i+1)B_0\delta^{\beta\wedge1}.$$
	  	 So the assertion is true for $i+1$. We take $\delta=3K^{-\beta\vee1}$, then
	  	\begin{equation*}
	  	 	\delta^{\beta\wedge1}=\Big( \frac{1}{3K^{\beta\vee1}}\Big)^{\beta\wedge1}=
	  	 \Big\{
	  	 	\begin{array}{cc}
	  	 		\frac{1}{3}K^{-\beta} & \beta\ge1, \\
	  	 		(3K)^{-\beta} & \beta<1,
	  	 	\end{array}
   	\end{equation*}
	  	 and $K=\lfloor(NM)^{2/d}\rfloor$. Since $E_d=[0,1]^d$, let $\phi:=\phi_d$, we have
	  	 \begin{align*}
	  	 	\Vert \phi-f\Vert_{L^\infty([0,1]^d)}\leq&\mathcal{E}+dB_0\delta^{\beta\wedge1}\\
	  	 	\leq& 18B_0(s+1)^2d^{s+(\beta\vee1)/2}(NM)^{-2\beta/d}+dB_0(NM)^{-2\beta/d}\\
	  	 	\leq&19B_0(s+1)^2d^{s+(\beta\vee1)/2}(NM)^{-2\beta/d},
	  	 \end{align*}
  	 where $s=\lfloor\beta\rfloor$, which completes the proof.
	  \end{proof}

\subsection{Proof of Theorem \ref{thm2}}
	\begin{proof}
	Let $K\in\mathbb{N}^+$ and $\delta\in(0,1/K)$, define a region $\Omega([0,1]^d,K,\delta)$ of $[0,1]^d$ as $$\Omega([0,1]^d,K,\delta)=\bigcup_{i=1}^d\{x=[x_1,x_2,...,x_d]^T:x_i\in\bigcup_{k=1}^{K-1}(k/K-\delta,k/K)\}.$$
	
	By Theorem \ref{thm_apx}, for any $M, N\in\mathbb{N}^+$, there exists a function
	$f^*_n\in\mathcal{F}_n=\mathcal{F}_{\mathcal{D},\mathcal{W},\mathcal{U},\mathcal{S},\mathcal{B}}$ with  width $\mathcal{W}=38(s+1)^2d^{s+1}N\lceil\log_2(8N)\rceil$ and depth $\mathcal{D}=21(s+1)^2M\lceil\log_2(8M)\rceil$,	
	such that
	$$ \vert f^*_n(x)-f_0(x)\vert\leq 18B_0(s+1)^2d^{s+(\beta\vee1)/2}(NM)^{-2\beta/d},$$
	for any $x\in[0,1]^d\backslash \Omega([0,1]^d,K,\delta)$ where $K=\lfloor N^{1/d}\rfloor^2\lfloor M^{2/d}\rfloor$ and $\delta$ is an arbitrary number in $(0,\frac{1}{3K}]$. Note that the Lebesgue measure of $\Omega([0,1]^d,K,\delta)$ is no more than $dK\delta$ which can be arbitrarily small if $\delta$ is arbitrarily small. Since $\nu$ is absolutely continuous with respect to the Lebesgue measure,  we have
	$$\Vert f^*_n-f_0\Vert^2_{L^2(\nu)}\leq18^2B_0^2(s+1)^4d^{2s+\beta\vee1}(NM)^{-4\beta/d}.$$
	By Lemma \ref{lemma1}, finally we have
	$$\mathbb{E} \Vert f^*_n-f_0\Vert^2_{L^2(\nu)}\leq C\mathcal{B}^2\frac{\mathcal{S}\mathcal{D}\log(\mathcal{S})(\log n)^3}{n}+324B_0^2(s+1)^4d^{2s+\beta\vee1}(NM)^{-4\beta/d},$$
	where $C$ does not depend on $n,d,N,M,s,\beta,B_0,\mathcal{D},\mathcal{B}$ or $\mathcal{S}$, and $s=\lfloor\beta\rfloor$. This completes the proof of Theorem \ref{thm2}.
		\end{proof}

	\subsection{Proof of Corollary \ref{c1}}
	We prove Corollary \ref{c1}. Corollaries \ref{c2} and \ref{c3} can be proved similarly.
	\begin{proof}
		Under the assumptions in Theorem \ref{thm2},  for any $N, M\in\mathbb{N}^+$, the function class of ReLU multi-layer perceptrons $\mathcal{F}_n=\mathcal{F}_{\mathcal{D},\mathcal{W},\mathcal{U},\mathcal{S},\mathcal{B}}$ with width $\mathcal{W}=38(s+1)^2d^{s+1}N\lceil\log_2(8N)\rceil$ and depth $\mathcal{D}=21(s+1)^2M\lceil\log_2(8M)\rceil$, the prediction error of the ERM
		$\hat{f}_n$ satisfies
		$$\mathbb{E} \Vert \hat{f}_n-f_0\Vert^2_{L^2(\nu)}\leq C\mathcal{B}^2
		(\log n)^3\, \frac{1}{n}\mathcal{S}\mathcal{D}\log(\mathcal{S}) +324B_0^2(s+1)^4d^{2s+\beta\vee1}(NM)^{-4\beta/d},$$
		for $2n \ge \text{Pdim}(\mathcal{F}_n)$, where $C>0$ is a constant not depending on $n,d,\mathcal{B},\mathcal{S},\mathcal{D},B_0$,\\$\beta,s,r,N$ or $M$.
		
		For deep with fixed width networks, given any $N\in\mathbb{N}^+$, the network width is fixed
		$$\mathcal{W}=38(s+1)^2d^{s+1}N\lceil\log_2(8N)\rceil.$$
		Recall that for any multilayer neural network in $\mathcal{F}_n$, its parameters naturally satisfy
		\begin{equation*}
			\max\{\mathcal{W},\mathcal{D}\}\leq
			\mathcal{S}\leq
			\mathcal{W}(d+1)+(\mathcal{W}^2+\mathcal{W})(\mathcal{D}-1)+\mathcal{W}+1\leq 2\mathcal{W}^2\mathcal{D}.
		\end{equation*}
		Then by plugging $\mathcal{S}\leq2\mathcal{W}^2\mathcal{D}$ and $\mathcal{D}=21(s+1)^2M\lceil\log_2(8M)\rceil$, we have
		\begin{align*}
			\mathbb{E} \Vert \hat{f}_n-f_0\Vert^2_{L^2(\nu)}&\leq C\mathcal{B}^2
			(\log n)^3\, \frac{1}{n}\mathcal{W}^2(M\lceil\log_2(8M)\rceil)^2\log(221(s+1)^2M\lceil\log_2(8M)\rceil\mathcal{W}^2)\\
			&\qquad +324B_0^2(s+1)^4d^{2s+\beta\vee1}(NM)^{-4\beta/d}.
		\end{align*}
		Note that the first term on the right hand side is increasing in $M$ while the second term is decreasing in $M$. To achieve the optimal rate with respect to $n$, we need a balanced choice of $M$ such  that
		$$(\log n)^3M^2\log(M)^2/n\approx M^{-4\beta/d},$$
		in terms of their order. This leads to the choice of $M=\lfloor n^{d/2(d+2\beta)}\rfloor$ and the network depth and size where
		\begin{eqnarray*}
			\mathcal{W}&=& 38(s+1)^2d^{s+1}N\lceil\log_2(8N)\rceil, \\
			\mathcal{S}&= &O(n^{d/2(d+2\beta)}(\log n)),
		\end{eqnarray*}
		the ERM $\hat{f}_n \in\arg\min_{f\in\mathcal{F}_n}L_n(f) $ satisfies
		\begin{align*}
			\mathbb{E} \Vert \hat{f}_n-f_0\Vert^2_{L^2(\nu)}\leq& \Big\{c_1\mathcal{B}^2(\log n)^{5}+324B_0^2d^{2s+\beta\vee1}N^{-4\beta/d}\Big\}(s+1)^4n^{-2\beta/(d+2\beta)},\\
			\leq&c_2\mathcal{B}^2N^{-4\beta/d}(s+1)^{4}d^{2s+\beta\vee1}n^{-2\beta/(d+2\beta)}(\log n)^{5},
		\end{align*}
		for $2n \ge \text{Pdim}(\mathcal{F}_n)$, where $c_1,c_2>0$ are constants which do not depend on $n,\mathcal{B},B_0,\beta,s$ or $N$. This completes the proof.
	\end{proof}

		{
	\subsection{Proof of Theorem \ref{thm3}}
	\begin{proof}
		We project the data to a low-dimensional space and then use DNN to do approximation the low-dimensional function where the idea  is similar to that of Theorem 1.2 in \cite{shen2019deep}.
		Based on Theorem 3.1 in \cite{baraniuk2009random}, there exists a linear projector $A\in\mathbb{R}^{d_\delta\times d}$ that maps a low-dimensional manifold in a high-dimensional space to a low-dimensional space nearly preserving the distance. Specifically, there exists a matrix $A\in\mathbb{R}^{d_\delta\times d}$ such that
		$AA^T=(d/d_\delta)I_{d_\delta}$ where $I_{d_\delta}$ is an identity matrix of size $d_\delta\times d_\delta$, and
		$$(1-\delta)\Vert x_1-x_2\Vert_2\leq\Vert Ax_1-Ax_2\Vert_2\leq(1+\delta)\Vert x_1-x_2\Vert_2,$$
		for any $x_1,x_2\in\mathcal{M}.$ And it is easy to check
		$$A(\mathcal{M}_\rho)\subseteq A([0,1]^d)\subseteq [-\sqrt{\frac{d}{d_\delta}},\sqrt{\frac{d}{d_\delta}}]^{d_\delta}.$$
		
		Note that for any $z\in A(\mathcal{M})$, there exists a unique $x\in\mathcal{M}$ such that $Ax=z$. To prove this, let $x^\prime\in\mathcal{M}$ be another point on $\mathcal{M}$ satisfying $Ax^\prime=z$, then $(1-\delta)\Vert x-x^\prime\Vert_2\leq\Vert Ax-Ax^\prime\Vert_2\leq(1+\delta)\Vert x-x^\prime\Vert_2$ implies that $\Vert x-x^\prime\Vert_2=0$.
		Then for any $z\in A(\mathcal{M})$, define $x_z=\mathcal{SL}(\{x\in\mathcal{M}: Ax=z\})$ where $\mathcal{SL}(\cdot)$ is a set function which returns a unique element of a set. Note that if $Ax=z$ where $x\in\mathcal{M}$ and $z\in A(\mathcal{M})$, then $x=x_z$ by our argument since $\{x\in\mathcal{M}: Ax=z\}$ is a set with only one element when $z\in A(\mathcal{M})$. And we can see that $\mathcal{SL}:A(\mathcal{M})\to\mathcal{M}$ is a differentiable function with the norm of its derivative locates in $[1/(1+\delta),1/(1-\delta)]$, since
		$$\frac{1}{1+\delta}\Vert z_1-z_2\Vert_2\leq\Vert x_{z_1}-x_{z_2}\Vert_2\leq\frac{1}{1-\delta}\Vert z_1-z_2\Vert_2,$$
		for any $z_1,z_2\in A(\mathcal{M})\subseteq E$ where $E:=[-\sqrt{{d}/{d_\delta}},\sqrt{{d}/{d_\delta}}]^{d_\delta}$.
		For the high-dimensional function $f_0: [0,1]^{d}\to\mathbb{R}^1$, we define its low-dimensional representation $\tilde{f}_0:\mathbb{R}^{d_\delta}\to\mathbb{R}^1$ by
		$$\tilde{f}_0(z)=f_0(x_z), \quad {\rm for\ any} \ z\in A(\mathcal{M})\subseteq\mathbb{R}^{d_\delta}.$$
		Recall that $f_0\in\mathcal{H}^\beta([0,1]^d,B_0)$, then $\tilde{f}_0\in\mathcal{H}^\beta(A(\mathcal{M}),B_0/(1-\delta)^\beta)$. Note that $\mathcal{M}$ is compact and $A$ is a linear mapping, then by the extended version of Whitney' extension theorem in \cite{fefferman2006whitney}, there exists a function $\tilde{F}_0\in\mathcal{H}^\beta(E,B_0/(1-\delta)^\beta)$ such that $\tilde{F}_0(z)=\tilde{f}_0(z)$ for any $z\in A(\mathcal{M})$. With $E=[-\sqrt{{d}/{d_\delta}},\sqrt{{d}/{d_\delta}}]^{d_\delta}$, by Theorem \ref{thm_apx}, for any $N,M\in\mathbb{N}^+$, there exists a function $\tilde{f}_n: \mathbb{R}^{d_\delta}\to\mathbb{R}^1$ implemented by a ReLU FNN with width $\mathcal{W}=38(s+1)^2d_\delta^{s+1}N\lceil\log_2(8N)\rceil$ and depth $\mathcal{D}=21(s+1)^2M\lceil\log_2(8M)\rceil$ such that
		$$\vert \tilde{f}_n(z)-\tilde{F}_0(z)\vert\leq 36\frac{B_0}{(1-\delta)^\beta}(s+1)^2d^{1/2}d_\delta^{3s/2}(NM)^{-2\beta/d_\delta},$$
		for all $z\in E\backslash\Omega(E)$ where $\Omega(E)$ is a subset of $E$ with an arbitrarily small Lebesgue measure as well as $\Omega:=\{x\in\mathcal{M}_\rho: Ax\in\Omega(E)\}$ does.
			
		If we define $f^*_n=\tilde{f}_n\circ A$ which is $f^*_n(x)=\tilde{f}_n(Ax)$ for any $x\in[0,1]^d$, then $f^*_n\in\mathcal{F}_{\mathcal{D},\mathcal{W},\mathcal{U},\mathcal{S},\mathcal{B}}$ is also a ReLU FNN with the same parameter as $\tilde{f}_n$. For any $x\in\mathcal{M}_\rho\backslash\Omega$ and $z=Ax$, there exists a $\tilde{x}\in\mathcal{M}$ such that $\Vert x-\tilde{x}\Vert_2\leq \rho$, then
		\begin{align*}
			&\vert f^*_n(x)-f_0(x)\vert=\vert \tilde{f}_n(Ax)-\tilde{F}_0(Ax)+\tilde{F}_0(Ax)-\tilde{F}_0(A\tilde{x})+\tilde{F}_0(A\tilde{x})-f_0(x)\vert\\
			&\leq\vert \tilde{f}_n(Ax)-\tilde{F}_0(Ax)\vert+\vert\tilde{F}_0(Ax)-\tilde{F}_0(A\tilde{x})\vert+\vert\tilde{F}_0(A\tilde{x})-f_0(x)\vert\\
			&\leq 36\frac{B_0}{(1-\delta)^\beta}(s+1)^2d^{1/2}d_\delta^{3s/2}(NM)^{-2\beta/d_\delta}+\frac{ B_0}{1-\delta}\Vert Ax-A\tilde{x}\Vert_2+\vert f_0(\tilde{x})-f_0(x)\vert\\
			&\leq 36\frac{B_0}{(1-\delta)^\beta}(s+1)^2d^{1/2}d_\delta^{3s/2}(NM)^{-2\beta/d_\delta}+\frac{\rho B_0}{1-\delta}\sqrt{\frac{d}{d_\delta}}+\rho B_0\\
			&=36\frac{B_0}{(1-\delta)^\beta}(s+1)^2d^{1/2}d_\delta^{3s/2}(NM)^{-2\beta/d_\delta}+\rho B_0\{(1-\delta)^{-1}\sqrt{{d}/{d_\delta}}+1\}\\
			&\leq (36+C_2)\frac{B_0}{(1-\delta)^\beta}(s+1)^2d^{1/2}d_\delta^{3s/2}(NM)^{-2\beta/d_\delta},
			\end{align*}
		where $C_2>0$ is a constant not depending on any parameter. The last inequality follows from $\rho\leq C_2(NM)^{-2\beta/d_{\delta}}(s+1)^2d^{1/2}d_\delta^{3s/2}\{\sqrt{{d}/{d_\delta}}+1-\delta\}^{-1}(1-\delta)^{1-\beta}$. Since  the probability measure $\nu$ of $X$ is absolutely continuous with respect to the Lebesgue measure, we have
		\begin{equation}\label{approxbound}
			\Vert f^*_n -f_0\Vert^2_{L^2(\nu)} \leq (36+C_2)^2\frac{B_0^2}{(1-\delta)^{2\beta}}(s+1)^4dd_\delta^{3s}(NM)^{-4\beta/d_\delta},
		\end{equation}
where $d_\delta=O(d_\mathcal{M}{\log(d/\delta)}/{\delta^2})$ is assumed
to satisfy $d_\delta \ll d$. By Lemma \ref{lemma1}, we have
\begin{align*}
		&\mathbb{E} \Vert \hat{f}_n-f_0\Vert^2_{L^2(\nu)}\\
&\leq C_1\mathcal{B}^2\frac{\mathcal{S}\mathcal{D}\log(\mathcal{S})
(\log n)^3}{n}+(36+C_2)^2\frac{B_0^2}{(1-\delta)^{2\beta}}(\lfloor\beta\rfloor+1)^4dd_\delta^{3\lfloor\beta\rfloor}(NM)^{-4\beta/d_\delta},
\end{align*}
where $C_1,C_2>0$ are constants that do not depend on $n,\mathcal{B},\mathcal{S},\mathcal{D},B_0,\beta,\delta,N$ or $M$, $\lfloor\beta\rfloor=s$ is the biggest integer strictly smaller than $\beta$.
This completes the proof of Theorem \ref{thm3}.

\end{proof}

{
\subsection{Proof of Theorem \ref{thm5}}
To facilitate the proof, we first briefly review manifolds, partition of unity, and function spaces defined on smooth manifolds. Details can be found in \citet{chen2019nonparametric}, \citet{tu2011manifolds},  \citet{lee2006riemannian}, \citet{federer1959curvature} and \citet{aamari2019estimating}.

\begin{definition}[Chart]
	Let $\mathcal{M}$ be a $d_{\mathcal{M}}$-dimensional Riemannian manifold isometrically embedded in $\mathbb{R}^d$. A chart for $\mathcal{M}$ is a pair $(U,\phi)$ such that $U\subset\mathcal{M}$ is open and $\phi: U\mapsto\mathbb{R}^{d_{\mathcal M}}$, where $\phi$ is a homeomorphism, i.e., bijective, $\phi$ and $\phi^{-1}$ are both continuous.
\end{definition}
We say two charts $(U,\phi)$ and $(V,\psi)$ on $\mathcal{M}$ are $C^k$ compatible if and only if the transition functions,
$$\phi\circ\psi^{-1}:\psi(U\cap V)\mapsto\phi(U\cap V)\quad{\rm and}\quad \psi\circ\phi^{-1}:\phi(U\cap V)\mapsto\psi(U\cap V)$$
are both $C^k$.

\begin{definition}[$C^k$ Atlas]
	A $C^k$ atlas for $\mathcal{M}$ is a collection of pairwise $C^k$ compatible charts $\{(U_i,\phi_i)\}_{i\in \mathcal{A}}$ such that $\bigcup_{i\in\mathcal{A}} U_i=\mathcal{M}$.
\end{definition}

\begin{definition}[Smooth manifold]
	A smooth manifold is a manifold together with a $C^\infty$ atlas.
\end{definition}

\begin{definition}[H\"older functions on $\mathcal{M}$]
	Let $\mathcal{M}$ be a $d_{\mathcal M}$-dimensional Riemannian manifold isometrically embedded in $\mathbb{R}^d$.  Let $\{(U_i,P_i)\}_{i\in \mathcal{A}}$ be an atlas of $\mathcal{M}$ where the $P_i's$ are orthogonal  projections onto tangent space. For a positive number $\beta>0$, a function $f:\mathcal{M}\mapsto\mathbb{R}$ belonging to H\"older class $\mathcal{H}^\beta(\mathcal{M},B_0)$ is $\beta$-H\"older smooth with constant $B_0$ if for each chart $(U_i,P_i)$ in the atlas, we have
	\begin{itemize}
		\item [1.] $f\circ P_i^{-1}\in C^s$ with  $\max_{\Vert\alpha\Vert_1\le s}\vert\partial^\alpha f (x)\vert\le B_0$ for any $x\in U_i$.
		\item [2.] For any $\Vert\alpha\Vert_1= s$ and $x,y\in U_i$,
		$$\sup_{x\not=y}\frac{\vert\partial^\alpha f(x)-\partial^\alpha f(y)\vert}{\Vert x-y\Vert_2^r}\le B_0,$$
		where $s$ is the largest integer strictly smaller than $\beta$ and $r=\beta-s$.
	\end{itemize}
\end{definition}

\begin{definition}[Partition of Unity, Definition 13.4 in \citet{tu2011manifolds}]
	A $C^\infty$ partition of unity on a manifold $\mathcal{M}$ is a collection of nonnegative $C^\infty$ functions $\rho_i:\mathcal{M}\mapsto\mathbb{R}^+$ for $i\in\mathcal{A}$ such that
		\begin{itemize}
		\item [1.] The collection of the supports, $\{{\rm supp}(\rho_i)\}_{i\in\mathcal{A}}$ is locally finite, i.e., every point on $\mathcal{M}$ has a neighborhood that meets only finitely many of ${\rm supp}(\rho_i)$'s.
		\item [2.] $\sum_{i\in\mathcal{A}} \rho_i=1$.
	\end{itemize}
\end{definition}
By Theorem 13.7 in \citet{tu2011manifolds}, a $C^\infty$ partition of unity always exists for a smooth manifold. This gives a decomposition $f=\sum_{i\in\mathcal{A}} f_i$ with $f_i=f\rho_i$ and each $f_i$ has the same regularity as $f$ since $f_i\circ\phi_i^{-1}=(f\circ\phi_i^{-1})\times(\rho_i\circ\phi_i^{-1})$ for a chart $(U_i,\phi_i)$. And the decomposition means that we can express $f$ as a sum of the $f_i$’s with each $f_i$ is only supported in a single chart.

Our approach builds on the methods of
\citet{schmidt2019deep, chen2019efficient} and \citet{chen2019nonparametric}
but there are some noteworthy new aspects:
(a) we apply linear coordinate maps instead of smooth coordinate maps,
 where the linear coordinate maps can be exactly
 represented by shallow ReLU  networks without error;
  (b)
   we do not require the smoothness index of each coordinate map and each function in the partition of unity to be no less than $\beta d/d_{\mathcal M}$,  which depends on the ambient dimension $d$ and can be large; (c) we apply our new approximation result when approximating the low-dimensional H\"older smooth functions on each projected chart, which leads to a better prefactor of error compared to most existing results


{\color{black}
\begin{proof}
We prove Theorem \ref{thm5} in  three steps: (1) we first construct an finite atlas that covers the manifold $\mathcal{M}$;
(2) we project each chart linearly to a $d_{\mathcal M}$-dimensional hypercube on which we approximate the low-dimensional H\"older smooth functions respectively; (3)  lastly, we combine the approximation results on all charts to get a error bound of the approximation on the whole manifold.

\textbf{Step 1:} Atlas Construction and Projection.\\
Let $B(x,r)$ denote the open Euclidean ball with radius $r>0$ and center $x\in\mathbb{R}^d$. Given any $r>0$, we have an open cover $\{B(x,r)\}_{x\in\mathcal{M}}$ of $\mathcal{M}$. By the compactness of $\mathcal{M}$, there exists a finite cover $\{B(x_i,r)\}_{i=1,\ldots,C_\mathcal{M}}$ for some finite integer $C_\mathcal{M}$ such that $\mathcal{M}\subset\bigcup_{i}B(x_i,r)$. Let $(1/\tau)$ denote the condition number of $\mathcal{M}$, then we can choose proper radius $r<\tau/2$ such that $U_i=\mathcal{M}\cap B(x_i,r)$ is diffeomorphic to a ball in $\mathbb{R}^{d_{\mathcal M}}$ \citep{niyogi2008finding}. The definition and detailed introduction of condition number (or its inverse called ``reach") can be found in \citet{federer1959curvature} and \citet{aamari2019estimating}. Besides, the number of charts $C_\mathcal{M}$ satisfies
$$C_{\mathcal{M}}\le\lceil S_{(\mathcal{M})}T_{d_{\mathcal M}}/r^{d_{\mathcal M}}\rceil,$$
where $S_{\mathcal{M}}$ is the area of the surface of $\mathcal{M}$ and $T_{d_{\mathcal M}}$ is the thickness of $U_i$'s, which is defined as the average number of $U_i$'s that contain a point on $\mathcal{M}$. By equation (19) in Chapter 2 of \citet{conway2013sphere}, the thickness $T_{d_{\mathcal M}}$ scales approximately linear in $d_{\mathcal M}$ and there exist coverings such that $T_{d_{\mathcal M}}\le d_{\mathcal M}\log(d_{\mathcal M})+d_{\mathcal M}\log\log(d_{\mathcal M})+5d_{\mathcal M}\le7d_{\mathcal M}\log(d_{\mathcal M})$. Let the tangent space of $\mathcal{M}$ at $x_i$ be denoted by $T_{x_i}(\mathcal{M})$ and let $V_i\in\mathbb{R}^{d\times d_{\mathcal M}}$ be the matrix concatenating the orthonormal basis of the tangent space as column vectors. Then for any $x\in U_i$ we can define the projection
$$\phi_i(x)=a_i(V_i^\top(x-x_i)+b_i),$$
where $a_i\in(0,1]$ and $b_i$ are proper scalar and vector such that $\phi_i(x)\in[0,1]^{d_{\mathcal M}}$ for any $x\in U_i$. Note that each projection $\phi_i$ is a linear function, which can be computed by a one-hidden layer ReLU network.

\textbf{Step 2:} Approximate low-dimensional functions.\\
For charts $\{(U_i,\phi_i)\}_{i=1}^{C_\mathcal{M}}$, we can approximate the function on each chart by approximation the projected function in the low-dimensional space. By Theorem 13.7 in \citet{tu2011manifolds}, the target function $f$ can be written as
$$f=\sum_{i=1}^{C_\mathcal{M}} f\rho_i:=\sum_{i=1}^{C_\mathcal{M}} f_i,$$
 where $\rho_i$'s are elements in $C^\infty$ partition of unity on $\mathcal{M}$ being supported in $U_i$'s.  Note that the manifold $\mathcal{M}$ is compact and smooth and $\rho_i$'s are $C^\infty$, then $f_i$'s have the same smoothness as $f$ itself for $i=1,\ldots,C_\mathcal{M}$. Note that the collection of the supports, $\{{\rm supp}(\rho_i)\}_{i\in\mathcal{A}}$ is locally finite, and let $C_\rho$ denote the maximum number of ${\rm supp}(\rho_i)$'s that a point on $\mathcal{M}$ can belong to.  Besides, since each $\phi_i$ is linear projection operator, it is not hard to show that each $f_i\circ\phi_i^{-1}$ is a H\"older smooth function with order $\beta>0$ on $\phi_i(U_i)\subset [0,1]^{d_{\mathcal M}}$, i.e., $f_i\circ\phi_i^{-1}\in\mathcal{H}^\beta(\phi_i(U_i),\sqrt{d/d_{\mathcal M}}B_0)$ for $i=1,\ldots,C_\mathcal{M}$. A detailed proof can be found in Lemma 2 of \citet{chen2019nonparametric}. By the extended version of Whitney' extension theorem in \cite{fefferman2006whitney}, we can approximate the smooth extension of $f_i\circ\phi_i^{-1}$ on $[0,1]^{d_{\mathcal M}}$. By Corollary \ref{thm_apx1}, for any $M,N\in\mathbb{N}^+$, there exists a function $g_i$ implemented by a ReLU network  with width $\mathcal{W}=38(\lfloor\beta\rfloor+1)^23^{d_{\mathcal M}}(d_{\mathcal M})^{\lfloor\beta\rfloor+1}N\lceil\log_2(8N)\rceil$ and depth $\mathcal{D}=21(\lfloor\beta\rfloor+1)^2M\lceil \log_2(8M)\rceil+2d_{\mathcal M}$ such that
$$\vert f_i\circ\phi_i^{-1}(x)-g_i(x) \vert\leq 19\sqrt{d/d_{\mathcal M}} B_0(\lfloor\beta\rfloor+1)^2(d_{\mathcal M})^{\lfloor\beta\rfloor+(\beta\vee1)/2}(NM)^{-2\beta/d_{\mathcal M}},$$
for any $x\in \phi_i(U_i)\subset[0,1]^{d_{\mathcal M}}$.

\textbf{Step 3:} Approximate the target function on the manifold.\\
By construction of subnetworks, the projected target functions $f_i\circ\phi_i^{-1}$ on each region $\phi_i(U_i)$ can be approximated by ReLU networks $g_i$. Note that each projection $\phi_i$ is a linear function can be computed by a one-hidden layer ReLU network. Then we stack two more layer to $g_i$ and get $\tilde{g}_i=g_i\circ\phi_i$ such that for any $x\in U_i$,
$$\vert f_i(x)-\tilde{g}_i(x)\vert=\vert f_i(x)-g_i\circ\phi_i(x) \vert\leq 19B_0(\lfloor\beta\rfloor+1)^2 d^{1/2}(d_{\mathcal M})^{\lfloor\beta\rfloor+\beta/2}(NM)^{-2\beta/d_{\mathcal M}},$$
where $\tilde{g}_i$ is a ReLU activated network with width $\mathcal{W}=38(\lfloor\beta\rfloor+1)^23^{d_{\mathcal M}}(d_{\mathcal M})^{\lfloor\beta\rfloor+1}N\lceil\log_2(8N)\rceil$ and depth $\mathcal{D}=21(\lfloor\beta\rfloor+1)^2M\lceil \log_2(8M)\rceil+2d_{\mathcal M}+2$. Since there are $C_\mathcal{M}$ charts, we parallel these subnetworks $\tilde{g}_i$ to get $\tilde{g}=\sum_{i=1}^{C_\mathcal{M}} \tilde{g}_i$ such that
\begin{align*}
\vert f(x)-\tilde{g}(x)\vert&=\left\vert \sum_{i=1}^{C_\mathcal{M}}f_i(x)-\sum_{i=1}^{C_\mathcal{M}} \tilde{g}_i(x)\right\vert\\
&\le C_\rho \max_{i=1,\ldots,C_\mathcal{M}}\vert f_i(x)- \tilde{g}_i(x)\vert\\
&\le19 C_\rho B_0(\lfloor\beta\rfloor+1)^2d^{1/2}(d_{\mathcal M})^{\lfloor\beta\rfloor+\beta/2}(NM)^{-2\beta/d_{\mathcal M}},
\end{align*}
for any $x\in\mathcal{M}$. Such a neural network $\tilde{g}$ has width $\mathcal{W}=38C_\mathcal{M}(\lfloor\beta\rfloor+1)^23^{d_{\mathcal M}}(d_{\mathcal M})^{\lfloor\beta\rfloor+1}N\lceil\log_2(8N)\rceil$ and depth $\mathcal{D}=21(\lfloor\beta\rfloor+1)^2M\lceil \log_2(8M)\rceil+2d_{\mathcal M}+2$. Recall that $C_{\mathcal{M}}\le\lceil S_{(\mathcal{M})}T_{d_{\mathcal M}}/r^{d_{\mathcal M}}\rceil\le \lceil 7S_{(\mathcal{M})}d_{\mathcal M}\log(d_{\mathcal M})/r^{d_{\mathcal M}}\rceil\le C_1 S_{(\mathcal{M})} (2/\tau)^{d_{\mathcal M}} d_{\mathcal M}\log(d_{\mathcal M})$ for some universal constant $C_1>0$, then width  $\mathcal{W}\le266(\lfloor\beta\rfloor+1)^2\lceil S_\mathcal{M}(6/\tau)^{d_{\mathcal M}}\rceil(d_{\mathcal M})^{\lfloor\beta\rfloor+2}N\lceil\log_2(8N)\rceil$. Then we have
\begin{align*}
	\vert f(x)-\tilde{g}(x)\vert &\le C_2 B_0 (\lfloor\beta\rfloor+1)^2d^{1/2}(d_{\mathcal M})^{3\lfloor\beta\rfloor/2+1/2}(NM)^{-2\beta/d_{\mathcal M}},
\end{align*}
where $C_2>0$ is some constant not depending on $n, d,d_{\mathcal M},N,M,\beta,B_0$ and $\tau$. And combining Lemma 3.2, we have
	$$\mathbb{E} \Vert \hat{f}_n-f_0\Vert^2_{L^2(\nu)}\leq C_1\mathcal{B}^2\frac{\mathcal{S}\mathcal{D}\log(\mathcal{S})
	(\log n)^3}{n}+C_3{B_0^2}(\lfloor\beta\rfloor+1)^4d(d_{\mathcal M})^{3\lfloor\beta\rfloor+1}(NM)^{-4\beta/d_{\mathcal M}},$$
where $C_2>0$ is some constant not depending on $n,d,d_{\mathcal M},\mathcal{B},\mathcal{S},\mathcal{D},N,M,\beta,B_0,\tau$ and $S_\mathcal{M}$.
This completes the proof of Theorem \ref{thm5}.
\end{proof}
}}

\subsection{Proof of Theorem \ref{thm4}}

\begin{proof}
	Let $E\subset\mathbb{R}^d$ be the support of $X$ with Minkowski dimension $d^*\equiv {\rm dim}_M(E)$. Let $T=\{(x_1-x_2)/\Vert x_1-x_2\Vert_2:x_1,x_2\in \bar{E}\}$ be the standardized difference of set $\bar{E}$ where $\bar{E}$ is the closure of $E$. By Lemma \ref{lemmab5}, there exists an absolute constant $\kappa$, a realization of random projection with entries i.i.d from Rademacher random variables $A:\mathbb{R}^d\to\mathbb{R}^{d_0}$ such that for all $\tau,\delta\in(0,1)$ if $d\ge d_0\ge \kappa(\gamma^2(T)+\log(2/\tau))/\delta^2$,
	$$(1-\delta)\Vert x_1-x_2\Vert_2^2\leq\Vert Ax_1-Ax_2\Vert_2^2\leq(1+\delta)\Vert x_1-x_2\Vert_2^2,$$
	for all $x_1,x_2\in \bar{E}$, where $\gamma(T)$ is defined in Lemma \ref{lemmab5}. Note that every covering (by closed balls) of $E$ is also a covering of $\bar{E}$, which implies ${\rm dim}_M(E)={\rm dim}_M(\bar{E})=d^*$. And $\gamma(T)$ is also related to $d^*$ the intrinsic dimension of $E$. More exactly, let $N_0=\mathcal{N}(\varepsilon,\Vert\cdot\Vert_2,\bar{E})$ be the covering number of $\bar{E}$ with radius $\varepsilon$ and $C_E=\{x_i\}_{i=1}^{N_0}\subset \bar{E}$ be the set of anchor points. Then for any $x\in\bar{E}$, there exists a $x_i$ such that $\Vert x-x_i\Vert_2\leq\varepsilon.$ For the difference set $\bar{E}-\bar{E}:=\{x_1-x_2:x_1,x_2\in\bar{E}\}$, we can construct a $2\varepsilon$-covering with $N_0^2$ anchor points $\{x_1-x_2:x_1,x_2\in C_E\}$. For any $y\in \bar{E}-\bar{E}$, there exists $x,x^\prime\in \bar{E}$ such that $y=x-x^\prime$. And there exists $x_1,x_2\in C_E$ such that $\Vert x-x_1\Vert_2\leq\varepsilon$ and $\Vert x^\prime-x_2\Vert_2\leq\varepsilon$. Then let $y^\prime=x_1-x_2$, we have $\Vert y-y^\prime\Vert_2=\Vert (x-x^\prime)-(x_1-x_2)\Vert_2\leq\Vert x-x_1\Vert_2+\Vert x^\prime-x_2\Vert_2\leq2\varepsilon.$ {\color{black} This shows that $\mathcal{N}(2\varepsilon,\Vert\cdot\Vert_2,\bar{E}-\bar{E})\leq\mathcal{N}(\varepsilon,\Vert\cdot\Vert_2,\bar{E})^2=N_0^2$, and ${\rm dim}_{M}(\bar{E}-\bar{E})\le2d^*$.
		Let $\bar{T}$ denote the bounded set $\bar{E}-\bar{E}$. Now we derive the relationship between the covering number of $\bar{T}$ and that of $T$. Firstly, given any real number $\delta>0$, we consider the subset $\bar{T}_\delta:=\{\bar{t}\in\bar{T}:\Vert\bar{t}\Vert_2\ge\delta\}$ and $T_\delta:=\{\bar{t}/\Vert\bar{t}\Vert_2:\bar{t}\in\bar{T}_\delta\}$. We scale up the set $\bar{T}_\delta:=\{\bar{t}\in\bar{T}_\delta:\Vert\bar{t}\Vert_2\ge\delta\}$ by $1/\delta$ times to get $\frac{1}{\delta}\bar{T}_\delta:=\{\bar{t}/\delta:\bar{t}\in\bar{T}_\delta\}$. By the definition of the Minkowski dimension (with respect the covering number) and the property of scaling, it is easy to see that the $\epsilon$-covering number of $\frac{1}{\delta}\bar{T}_\delta$ is no more than $(1/\delta)^{2d^*}$ times larger than that of $\bar{T}_\delta$ since ${\rm dim}_{M}(\bar{T}_\delta)\le{\rm dim}_{M}(\bar{T})\le2d^*$, i.e. for each $\delta>0$ we have,
		\begin{align*}
			\mathcal{N}(\epsilon,\Vert\cdot\Vert_2,\frac{1}{\delta}\bar{T}_\delta)&\le (1/\delta)^{2d^*}\mathcal{N}(\epsilon,\Vert\cdot\Vert_2,\bar{T}_\delta)\\
			&\le(1/\delta)^{2d^*}\mathcal{N}(\epsilon,\Vert\cdot\Vert_2,\bar{T})\\
			&\le c_0(1/\delta)^{2d^*}(1/\epsilon)^{2d^*},
		\end{align*}
		where $c_0>0$ is a constant not depending on $d^*,\epsilon$ and $\delta$. This implies ${\rm dim}_M(\frac{1}{\delta}\bar{T}_\delta)\le2d^*+{\rm dim}_M(\bar{T}_\delta)$ for $\delta>0$.
		
		Now $\delta>0$, we link the Minkowski dimension of $\frac{1}{\delta}\bar{T}_\delta$ to that $T_\delta$. Given any $\epsilon$, suppose $\bar{t}_1,\ldots,\bar{t}_m$ are the anchor points of a minimal $\epsilon$-cover of $\frac{1}{\delta}\bar{T}_\delta$. By the definition of covering, for any $\bar{t}\in\frac{1}{\delta}\bar{T}_\delta$, there exists an anchor point $\bar{t}_i$ for some $i\in\{1,\ldots,m\}$ such that $\Vert\bar{t}-\bar{t}_i\Vert_2\le\epsilon$. Since $\bar{t},\bar{t}_i\in\frac{1}{\delta}\bar{T}_\delta$, we have $\Vert\bar{t}\Vert_2\ge1,\Vert\bar{t}_i\Vert_2\ge1$, and
		\begin{align*}
			\Big\Vert\frac{\bar{t}}{\Vert\bar{t}\Vert_2}-\frac{\bar{t}_i}{\Vert\bar{t}_i\Vert_2}\Big\Vert_2&\leq\Big\Vert\frac{\bar{t}}{\Vert\bar{t}\Vert_2}-\frac{\bar{t}}{\Vert\bar{t}_i\Vert_2}+\frac{\bar{t}}{\Vert\bar{t}_i\Vert_2}-\frac{\bar{t}_i}{\Vert\bar{t}_i\Vert_2}\Big\Vert_2\\
			&\leq\Big\vert\frac{1}{\Vert\bar{t}\Vert_2}-\frac{1}{\Vert\bar{t}_i\Vert_2}\Big\vert\Vert\bar{t}\Vert_2+\frac{\Vert\bar{t}-\bar{t}_i\Vert_2}{\Vert\bar{t}_i\Vert_2}\\
			&\leq\Big\vert\frac{\epsilon}{\Vert\bar{t}\Vert_2(\Vert\bar{t}\Vert_2+\epsilon)}\Big\vert\Vert\bar{t}\Vert_2+\frac{\epsilon}{\Vert\bar{t}_i\Vert_2}\\
			&\leq2\epsilon.
		\end{align*}
		Thus the ball around ${\bar{t}_i}/{\Vert\bar{t}_i\Vert_2}$ with radius $2\epsilon$ covers ${\bar{t}}/{\Vert\bar{t}\Vert_2}$, which implies $\mathcal{N}(2\epsilon,\Vert\cdot\Vert_2,T_\delta)\le\mathcal{N}(\epsilon,\Vert\cdot\Vert_2,\frac{1}{\delta}\bar{T}_\delta)$. Then ${\rm dim}_M(T_\delta)\le{\rm dim}_M(\frac{1}{\delta}\bar{T}_\delta)\le 2d^*+{\rm dim}_M(\bar{T}_\delta)\le 2d^*+{\rm dim}_M(\bar{T})\le4d^*$. Since $\lim_{\delta\downarrow0}\bar{T}_\delta=\bar{T}$ and $\lim_{\delta\downarrow0}{T}_\delta={T}$ are both bounded, then
		\begin{align*}
			\sqrt{H(\varepsilon,\Vert\cdot\Vert_2,T)}&=\sqrt{\log(\mathcal{N}(\varepsilon,\Vert\cdot\Vert_2,T))}\leq c_1\sqrt{d^*\log(1/\varepsilon)},
		\end{align*}
		for some constant $c_1>0$.}
	Then by the definition of $\gamma(T)=\int_0^1\sqrt{H(\varepsilon,\Vert\cdot\Vert_2,T)}d\varepsilon$, we know $\gamma^2(T)=cd^*$ for some constant $c>0$. And $d_0\ge \kappa(\gamma^2(T)+\log(2/\tau))/\delta^2=\kappa(cd^*+\log(2/\tau))/\delta^2$.

	Since each entry of $A$ is either $1$ or $-1$, then $A(E)\subseteq A([0,1]^d)\subseteq H:=[-\sqrt{dd_0},\sqrt{dd_0}]^{d_0}$.
	Note that for any $z\in A(\bar{E})$, there exists a unique $x\in\bar{E}$ such that $Ax=z$. To prove this, let $x^\prime\in\bar{E}$ be another point in $\bar{E}$ satisfying $Ax^\prime=z$, then $(1-\delta)\Vert x-x^\prime\Vert^2_2\leq\Vert Ax-Ax^\prime\Vert_2^2\leq(1+\delta)\Vert x-x^\prime\Vert^2_2$ implies that $\Vert x-x^\prime\Vert_2=0$.
	Then we can define a one-one map $\mathcal{SL}$ from $A(\bar{E})$ to $\bar{E}$, i.e. $x_z=\mathcal{SL}(\{x\in\bar{E}: Ax=z\})$. And we can see that $\mathcal{SL}:A(\bar{E})\to\bar{E}$ is a differentiable function with the norm of its derivative locates in $[\sqrt{1/(1+\delta)},\sqrt{1/(1-\delta)}]$, since
	$$\frac{1}{1+\delta}\Vert z_1-z_2\Vert_2^2\leq\Vert x_{z_1}-x_{z_2}\Vert_2^2\leq\frac{1}{1-\delta}\Vert z_1-z_2\Vert_2^2,$$
	for any $z_1,z_2\in A(\bar{E})$.
	For the high-dimensional function $f_0: [0,1]^{d}\to\mathbb{R}^1$, we define its low-dimensional representation $\tilde{f}_0:\mathbb{R}^{d_0}\to\mathbb{R}^1$ by
	$$\tilde{f}_0(z)=f_0(x_z), \quad {\rm for\ any} \ z\in A(\bar{E})\subseteq\mathbb{R}^{d_0},$$
	with $\tilde{f}_0\in\mathcal{H}^\beta(A(\bar{E}),B_0/(1-\delta)^{\beta/2})$ recalling that $f_0\in\mathcal{H}^\beta([0,1]^d,B_0)$.  By the extended version of Whitney' extension theorem in \cite{fefferman2006whitney}, there exists a function $\tilde{F}_0\in\mathcal{H}^\beta(H,B_0/(1-\delta)^{\beta/2})$ such that $\tilde{F}_0(z)=\tilde{f}_0(z)$ for any $z\in A(\bar{E})$. With $H=[-\sqrt{dd_0},\sqrt{dd_0}]^{d_0}$, by Corollary \ref{thm_apx1}, for any $N,M\in\mathbb{N}^+$, there exists a function $\tilde{f_n}: \mathbb{R}^{d_0}\to\mathbb{R}^1$ implemented by a ReLU FNN with width $\mathcal{W}=38(\lfloor\beta\rfloor+1)^23^{d_0}d_0^{\lfloor\beta\rfloor+1}N\lceil\log_2(8N)\rceil$ and depth $\mathcal{D}=21(\lfloor\beta\rfloor+1)^2M\lceil \log_2(8M)\rceil+2d_0$ such that
	$$\vert \tilde{f}_n(z)-\tilde{F}_0(z)\vert\leq c_2\frac{B_0}{(1-\delta)^{\beta/2}}(\lfloor\beta\rfloor+1)^2d^{1/2}d_0^{\lfloor\beta\rfloor+(\beta\vee1+1)/2}(NM)^{-2\beta/d_0},$$
	for any $z\in H$ where $c_2$ is a constant not depending on $d,d_0,\beta,N$ or $M$.
	If we define $f^*_n=\tilde{f}_n\circ A$ which is $f^*_n(x)=\tilde{f}_n(Ax)$ for any $x\in[0,1]^d$, then $f^*_n\in\mathcal{F}_{\mathcal{D},\mathcal{W},\mathcal{S},\mathcal{B}}$ is also a ReLU FNN with the same parameter as $\tilde{f}_n$. For any $x\in{E}$,
	\begin{align*}
		\vert f^*_n(x)-f_0(x)\vert&=\vert \tilde{f}_n(Ax)-\tilde{F}_0(Ax)\vert \\
		&\leq c_2\frac{B_0}{(1-\delta)^{\beta/2}}(\lfloor\beta\rfloor+1)^2d^{1/2}
		d_0^{\lfloor\beta\rfloor+(\beta\vee1+1)/2}(NM)^{-2\beta/d_0}.
	\end{align*}
	Combining with Lemma \ref{lemma1}, since $X$ is supported on $E$, we have
	$$\mathbb{E} \Vert \hat{f}_n-f_0\Vert^2_{L^2(\nu)}\leq C_1\mathcal{B}^2\frac{\mathcal{S}\mathcal{D}\log(\mathcal{S})
		(\log n)^3}{n}+C_2\frac{B_0^2}{(1-\delta)^{\beta}}(\lfloor\beta\rfloor+1)^4dd_0^{2\lfloor\beta\rfloor+\beta\vee1+1}(NM)^{-4\beta/d_0},$$
	where $d_0\ge \kappa d^*/\delta^2=O(d^*/\delta^2)$ for some constants $\kappa>0$ and $C_1,C_2>0$ are constants that do not depend on $n,d,d_0\mathcal{B},\mathcal{S},\mathcal{D},B_0,\beta,\kappa,\delta,N$ or $M$, $\lfloor\beta\rfloor=s$ is the biggest integer strictly smaller than $\beta$.
	This completes the proof of Theorem \ref{thm4}.
\end{proof}
}

 {
 \section{Supporting Lemmas}
For ease of reference, we collect several existing results
 that we used in our proofs.

\begin{lemma}[Proposition 4.3. in \cite{lu2020deep}]\label{lemmab1}
	For any $N,M,d\in\mathbb{N}^+$ and $\delta\in(0,3K]$ with $K=\lfloor N^{1/d}\rfloor^2\lfloor M^{2/d}\rfloor$, there exists a one-dimensional function $\phi$ implemented by a ReLU FNN with width $4\lfloor N^{1/d}\rfloor+3$ and depth $4M+5$ such that
	$$\phi(x)=k,\quad {\rm if\ }x\in[\frac{k}{K},\frac{k+1}{K}-\delta\cdot1_{k<K-1}],\ {\rm for\ } k=0,1,\ldots,K-1.$$
\end{lemma}


\begin{lemma}[Proposition 4.4. in \cite{lu2020deep}]\label{lemmab2}
	Given any $N,M,s\in\mathbb{N}^+$ and $\xi_i\in[0,1]$ for $i=0,1,\ldots,N^2L^2-1$, there exists a function $\phi$ implemented by a ReLU FNN with width $16s(N+1)\lceil\log_2(8N)\rceil$ and depth $5(M+2)\lceil\log_2(4M)\rceil$ such that
	$$\vert\phi(i)-\xi_i\vert\leq N^{-2s}M^{-2s},\ {\rm for\ }i=0,1,\ldots,N^2M^2-1,$$
	and $0\le\phi(x)\le1$ for any $x\in\mathbb{R}$.
\end{lemma}

The next lemma demonstrate that the production function and polynomials can be approximated by ReLU neural networks. The basic idea is firstly to approximate the square function using ``sawtooth" functions then the production function, which is firstly raised in \cite{yarotsky2017error}. A general polynomial can be further approximated combining the approximated square function and production function. The following two lemmas are more general results than those in \cite{yarotsky2017error}.
\begin{lemma}[Lemma 4.2. in \cite{lu2020deep}]\label{lemmab3}
	For any $N,M\in\mathbb{N}^+$, and $a,b\in\mathbb{R}$ with $a<b$, there exists a function $\phi$ implemented by a ReLU FNN with width $9N+1$ and depth $M$ such that
	$$\vert\phi(x,y)-xy\vert\le 6(b-a)^2N^{-M}$$
	for any $x,y\in[a,b]$.
\end{lemma}

\begin{lemma}[Theorem 4.1 in \cite{lu2020deep}]\label{lemmab4}
	Assume $P(x)=x^\alpha=x_1^{\alpha_1}x_2^{\alpha_2}\cdots x_d^{\alpha_d}$ for $\alpha\in\mathbb{N}^d$ with $\Vert\alpha\Vert_1\leq k\in\mathbb{N}^+$. For any $N,M\in\mathbb{N}^+$, there exists a function $\phi$ implemented by a ReLU FNN with width $9(N+1)+k-1$ and depth $7k^2M$ such that
	$$\vert\phi(x)-P(x)\vert\leq9k(N+1)^{-7kM},\quad{\rm for\ any\ }x\in[0,1]^d.$$
\end{lemma}

The next lemma is a generalization of Johnson-Lindenstrauss theorem for embedding a set with infinitely many elements, which is firstly proved in \cite{klartag2005empirical}.
\begin{lemma}[Theorem 13.15 in \cite{boucheron2013concentration}] \label{lemmab5}
	Let $A\subset \mathbb{R}^d$ and consider the random projection  $W:\mathbb{R}^d\to\mathbb{R}^{d_0}$ with its entries are independent either standard Gaussian or Rademacher random variables. Let $T=\{(a_1-a_2)/\Vert a_1-a_2\Vert_2:a_1,a_2\in A\}$ and define
	$$\gamma(T)=\int_0^1\sqrt{H(\varepsilon,\Vert\cdot\Vert_2,T)}d\varepsilon,$$
	where $H(\varepsilon,\Vert\cdot\Vert_2,T)$ is the $\varepsilon$-entropy of $T$ with respect to the norm $\Vert\cdot\Vert_2$. There exists an absolute constant $\kappa^\prime$, such that for all $\tau,\delta\in(0,1)$ if $d_0\geq\kappa^\prime(\gamma^2(T)+\log(2/\tau))/\delta^2$, then with probability at least $1-\tau$,
	$$(1-\delta)\Vert a_1-a_2\Vert_2^2\leq\Vert Wa_1-Wa_2\Vert_2^2\leq(1+\delta)\Vert a_1-a_2\Vert_2^2,$$
	for all $a_1,a_2\in A$.
\end{lemma}

}

\end{appendix}

\bibliographystyle{imsart-nameyear} 
\bibliography{dnr_bib}    

\end{document}